\documentclass[11pt, a4paper,leqno]{amsart}
 
\usepackage{graphicx}
\usepackage{amsfonts}
\usepackage{amsmath,amssymb,amsthm,amscd,cancel,multicol,stackrel, graphics}
\usepackage{arydshln}
\usepackage{memhfixc}
\usepackage{latexsym}
\usepackage{tikz}
\usepackage{enumerate}
\usepackage{esint}
\usepackage{url}
\usepackage[colorlinks=true, pdfstartview=FitV, linkcolor=blue, citecolor=red, urlcolor=blue]{hyperref} 

\calclayout
\allowdisplaybreaks
\newcommand{\mfa}{\mathfrak{a}}
\newcommand{\mfb}{\mathfrak{b}}
\newcommand{\mfc}{\mathfrak{c}}
\newcommand{\al}{\alpha}
\newcommand{\be}{\beta}
\newcommand{\ga}{\gamma}

\DeclareMathOperator{\codim}{codim}

\DeclareMathOperator{\lcm}{lcm}

\DeclareMathOperator{\projdim}{proj\,dim}

\DeclareMathOperator{\reg}{reg}

\DeclareMathOperator{\pnt}{\raise 0.5mm \hbox{\large\bf.}}

\theoremstyle{plain}
\newtheorem{theorem}[equation]{Theorem}
\newtheorem{lemma}[equation]{Lemma}
\newtheorem{corollary}[equation]{Corollary}
\newtheorem{proposition}[equation]{Proposition}
\newtheorem{conjecture}[equation]{Conjecture}

\newtheorem{algorithm}[equation]{Algorithm}

\theoremstyle{definition}
\newtheorem{definition}[equation]{Definition}

\newtheorem{problem}[equation]{Problem}
\newtheorem{example}[equation]{Example}

\theoremstyle{remark}
\newtheorem{remark}[equation]{Remark}

\numberwithin{equation}{section}

 \newcount\HOUR
 \newcount\MINUTE
 \newcount\HOURSINMINUTES
 \newcount\INTVAL
 \newcommand{\twodigit}[1]{\INTVAL=#1\relax\ifnum\INTVAL<10 0\fi\the\INTVAL}
 \HOUR=\time\divide\HOUR by 60\relax
 \HOURSINMINUTES=\HOUR\multiply\HOURSINMINUTES by 60\relax
 \MINUTE=\time\advance\MINUTE by -\HOURSINMINUTES\relax
 \newcommand\rightnow{
             \twodigit{\the\HOUR}:\twodigit{\the\MINUTE},
           \twodigit{\number\day}.\space
             \ifcase\month\or January\or February\or March\or April\or 
 May\or June\or July\or August\or September\or October\or November\or 
 December\fi
            \space\number\year}

\title{Boij-S\"{o}derberg Decompositions of Lex-segment Ideals} 

\author{Sema G\"unt\"urk\"un}
\email{gunturku@umich.edu}
\address{Department of Mathematics,
University of Michigan,
2074 East Hall, 530 Church Street
Ann Arbor, MI  48109-1043\\
 USA}

\keywords{Boij-S\"{o}derberg decompositions, Betti diagrams, Lex-segment ideals}
\subjclass[2010]{13D02, 13C99, 13A99}

\begin{document}
\allowdisplaybreaks

\begin{abstract}
Boij-S\"oderberg theory describes the scalar multiples of Betti diagrams of graded modules over a polynomial ring as a linear combination of pure diagrams 
with positive coefficients. There are a few results that describe Boij-S\"oderberg decompositions explicitly. In this paper, we focus on the Betti diagrams of lex-segment ideals. Mainly, we characterize the Boij-S\"oderberg decomposition of a lex-segment ideal and describe it by using Boij-S\"oderberg decompositions of some other related lex-segment ideals. 
 \end{abstract}

\maketitle 

\section{Introduction}
\label{intro}
Boij-S\"oderberg is very recent theory which addresses the characterization of Betti diagrams of graded modules in polynomial rings.
It is originated in a pair of conjectures by Boij and S\"oderberg in \cite{BS} and in \cite{BS2} whose proof is given by Eisenbud and Schreyer in \cite{ES}. 
Their result gives a characterization of Betti diagrams of graded modules up to scalar multiples. For more information about Boij-S\"oderberg theory, we refer to a very informative survey written by Fl{\o}ystad  in \cite{F}. This theory brings up an idea of decomposition of the Betti diagrams of graded modules whose resolutions are not pure resolution. If the resolution is pure, then the decomposition consists of only one pure diagram with a positive coefficient as expected. There is not much known 
about the behavior of the Boij-S\"oderberg decomposition of an ideal in a polynomial ring.  Any characterization of Boij-S\"oderberg decompositions, either about Boij-S\"oderberg coefficients or  about the chain of the degree sequences associated with the pure diagrams, would also assist us in understanding and interpreting the  structural consequences of Boij-S\"oderberg decomposition of the Betti diagrams. Although the theory is quite recent and has a lot of open problems, improvements and contributions to this theory are quite impressive. Cook in \cite{Cook} and Berkesch, Erman, Kumini, and Sam in \cite{BerErmKumSam} discuss Boij-S\"oderberg theory in the perspective of poset structures. In \cite{NaS}, Nagel and Sturgeon examine the Boij-S\"oderberg decomposition of some ideals that are raised from some combinatorial objects. They show the combinatorial importance of the coefficients of the pure diagrams in the Boij-S\"oderberg decompositions of their interest of ideals. In \cite{GJMRSW}, results of  Gibbons, Jeffries, Mayes, Rauciu, Stone and White provide a relation between decomposition of the Betti diagrams of complete intersections and degrees of their minimal generators. Another recent work is done by Francisco, Mermin, and Schweig in \cite{FMS}. In their paper, they study the behavior of the Boij-S\"oderberg coefficients of Borel ideals.

For the sake of simplicity, the abbreviation BS is used for the Boij-S\"oderberg . In this paper, we study the behavior of BS decompositions of lex-segment ideals and obtain a neat relation between the BS decompositions of a given lex ideal and some other related lex ideals. Throughout the paper, our main focus will be the BS chain of the degree sequences in the decomposition and we also provide a strong correlation of the coefficients of the pure diagrams as well. The reason of why we are interested in the BS decomposition of lex-segment ideals is based on the fact that lex ideals have very particular Betti diagrams. The Bigatti-Hulett-Pardue, in \cite{Bi,Hu,Pa},  prove that the lex-segment ideals have the largest Betti 
numbers among the ideals with the same Hilbert function. This pivotal property of lex-segment ideals makes their BS decompositions worthy of study. Moreover, Eliahou-Kervaire formula gives a nice formulation for the Betti diagram of lex ideals. The main goal is to obtain a pattern for the BS decomposition of a lex ideal by using decompositions of some other related lex-segment ideals.

In what follows, let $R = \mathbf{k}[x,y,z]$ be a polynomial ring of $3$ variables, with the lexicographic order, $x >_{lex} y >_{lex} z$ and  $L$ be a lex-segment ideal in $R$. The ideal $L$ can be decomposed as $L = x\mfa+J$ where $\mfa$ is also a lex-segment ideal in $R$ and $J$ is a lex-segment ideal in $\mathbf{k}[y,z]$. The first main result of this paper describes the ``beginning" of the BS decomposition of $L$ in terms of the decomposition of $\mfa$. The algorithm of BS decomposition itself provides a chain of degree sequences. The first degree sequence in the chain is the top degree sequence of the Betti diagram of $L$. By the algorithm, the second degree sequences is the top degree sequence of the remaining diagram after the subtraction of the first pure diagram with a suitable coefficient from the Betti diagram. It continues until the Betti diagram is decomposed completely. 
Thus, by saying that the ``beginning" of the BS decomposition, we mean the several degree sequences, that is pure diagrams, are obtained in the beginning of the BS decomposition of $L$. We now state our first result.
\begin{theorem}\label{thm1} 
Let $R = \mathbf{k}[x,y,z]$ and $L$ be a lex-segment ideal of codimension $3$ in $R$. Suppose $1\neq \mfa = L:(x)$. Write the Boij-S\"{o}derberg decomposition of $\mfa$ as
\[\be(\mfa) = \sum\limits_{i=0}^{t} \al_i \pi_{{\bf d^i}}  + R_{\mfa},\] 
where  $\bf d^0 < d^1 < ...<d^l<... < d^t$ are all top degree sequences  of length $3$, that is,  
${\bf d^i} = (d_0^i, d_1^i,d_2^i)$ for $i = 0,1,...,t$, and $R_{\mfa}$ is the linear combination of the pure diagrams greater that $\pi_{{\bf d^t}}$. Then the Boij-S\"oderberg decomposition of $L$ has the form
\[\be(L) = \sum\limits_{i=0}^{t} \tilde{\al}_i \pi_{{\bf \bar{d}^i}}  + R_{L}
\] 
where
${\bf \bar{d}^i} = {\bf d^i + 1} = (d_0^i +1, d_1^i +1,d_2^i +1)$,  and $\tilde{\al}_i = \al_i$ \ for $i = 0,1,...,t$ 
and $\tilde{\al}_t\geq\al_t$, and $R_{L}$ is a linear combination of pure diagrams greater than $\pi_{\bf \bar{d}^t}$.
%
\end{theorem}
The second main result in this article is devoted to the pure diagrams - that is to say, degree sequences - of the BS decomposition of the Betti diagrams of $L$ and $(J,x)=(L,x)$ in the polynomial ring $R = \mathbf{k}[x,y,z]$. Like in Theorem \ref{thm1}, we notice some similarities of the BS decompositions of lex ideal $L$ and $(L,x)$. We reveal that the entire part of the  BS decomposition of $(L,x)$ containing all pure diagrams of length less than $3$ shows up precisely 
at the end of the BS decomposition of $L$. That is, all pure diagrams of length less than $3$ are exactly the same with the coefficients. In particular, we prove that
\begin{theorem}\label{thm2} Let $L \subset R = \mathbf{k}[x,y,z]$ be an Artinian lex-segment ideal of codimension $3$.  
Suppose that $L$ cannot be decomposed as $ L = x(x,y,z^t) + J$ where $J$ is different from $(y,z)^{G_{min}(J)}$ and $1 < t < k-1$.

Let $\mfa = L : (x)$ be a lex-segment ideal of $R$. Then $L = x\mfa + J$ where $J \in \mathbf{k}[y,z]$ is a stable ideal of $\codim 2$.
The ideal $(J,x)=(L,x)$ is also a $\codim 3$ Artinian, lex-segment ideal in $R$.
\[\be(L,x) =  R_{(L,x)} + \sum\limits_{i=t+1}^{n} \al_i \pi_{{\bf d^i}}  \] 
where  ${\bf d}^{t+1} < {\bf d}^{t+2} < ... < {\bf d}^{n}$ are all top degree sequences  of length less than $3$,  
with the coefficients $\al_i$, \ $i = t+1,...,n$. $R_{(L,x)}$ is the linear combination of the pure diagrams associated with the degree sequences of length $3$. 

Then the BS decomposition of $L$ is
\[\be(L) =  R_{L} + \sum\limits_{i=t+1}^{n} \al_i \pi_{{\bf d^i}}  \]
where the chain ${\bf d^{t+1} < d^{t+2} < ... < d^{n}}$  of degree sequences of length $2$ and $1$ exactly with the same coefficients $\al_i$ and $R_L$ is the linear
combination of the pure diagrams associated with the degree sequences of length $3$.
\end{theorem}

As a plan of this paper, we first discuss some useful relations of the Betti numbers of the ideals $L$, $\mfa$ and $J$ in section \ref{results}. We also describe the entire Betti diagram of the lex ideal $L$ in terms of the Betti numbers of the 
colon ideal $\mfa = L : (x)$ and the stable ideal $J$ in the same section.  In section \ref{thm1section}, we give the proof of Theorem \ref{thm1} which gives the relation between the beginning of the BS decompositions of $L$ and $\mfa$. Proof of Theorem \ref{thm2} is given in section \ref{thm2newsection}.


Combining the results of Theorems \ref{thm1} and \ref{thm2}, in the case of $R = \mathbf{k}[x,y,z]$, 
 we give the following diagram to summarize nicely the relation between the degree sequences in the BS decompositions of $L$ and the ideals $\mfa$ and $(L,x)$.
\[
\begin{matrix}\text{The chain of  }\\ \text{ degree sequences } \\ \text{ of $L$} \end{matrix} = \left[ \begin{matrix}\text{all length $3$} \\ \text{ degree sequences  }\\ \text{ coming from } \\ \mfa(-1) \end{matrix} \right]  
< \left[ \begin{matrix}\text{more length $3$}\\ \text{degree sequences} \\ \text{ (section \ref{observation})} \end{matrix}\right] < 
\left[ \begin{matrix} \text{all length $<3$  } \\ \text{ degree sequences } \\ \text{ coming from } \\ (L,x) \end{matrix} \right].
\]
We will see that, most of the time, the decompositions of $\mfa$ and $(L,x)$  may not enough to cover all pure diagrams in the decomposition of $L$ since there might be some pure diagrams of length $3$ which may not obtained by the ideal $\mfa$.


One naturally hopes to obtain a description of the entire BS decomposition of lex-segment ideal $L$ in terms of some other related ideals. Section \ref{observation} includes further observations for a possible way to describe the entire BS chain of degree sequences in the BS decomposition of $L$. 
The lexicographic order $x>_{lex}y>_{lex}z$ makes us think about the colon ideals $\mfb = L : (y)$ and $\mfc = L : (z)$. As in the case $\mfa = L : (x)$ in section \ref{thm1section}, one may expect similar results for the lex ideals $\mfb$ and $\mfc$. Indeed, the examples show that there is a relation between the BS decompositions of the lex ideal $L$ and the colon ideals $\mfb$ and $\mfc$. This allows us to give an almost full description of the pure diagrams appearing in the decomposition of $L$.

\section{Background and Preliminaries}
\label{results}
Throughout this section we assume that $R$ is a graded polynomial ring with $n < \infty$ variables over a field $\mathbf{k}$ with each variable has degree one. 
In the case of $n=3$, we will see the description of the Betti diagram $L = x\mfa + J$ in terms of the Betti numbers of $\mfa$ and $J$.

Let $M$ be a graded $R$-module. The minimal graded free resolution of $M$ is written as 
\[
\mathbb{F}:\, \, 0 \to F_n \rightarrow \ldots \rightarrow F_i \rightarrow \ldots \rightarrow F_1 \to F_0 \rightarrow M \rightarrow 0 
\]
where
\[
F_i = \bigoplus\limits_{\alpha\geq 0} R(-\alpha)^{\be_{i,\alpha}}.
\]
The numbers $\be_{i,\alpha}$ are the Betti numbers of $M$ and are considered in the Betti diagram 
$\be(M)$ of $M$ whose entry in row $i$ and column $j$ is $\be_{i,i+j}$. Let 
$
{\bf d} = (d_0, d_1,..., d_{n-1}) \in \mathbb{Z}_{\geq 0 }^{n}
$ 
be a sequence of non-negative integers of length $n+1$ with $d_0<...<d_{n-1}$. 
The graded free resolution of $M$ is called a \textit{pure
resolution of type ${\bf d} = (d_0,...,d_i,...,d_{n-1})$ } if, for all $i = 0,1,...,n-1$, the $i$-th syzygy module of $M$ is generated only by elements of degree 
$d_i$, in other words, all Betti numbers are zero
except $\be_{i,d_i}(M)$. Then the Betti diagram of this module is called a \textit{pure diagram of type ${\bf d}$}. The formula for 
the pure diagram associated by ${\bf d}$ is based on the Herzog-K\"uhl equations introduced in \cite{HK},
\[ \be_{i,\alpha}  = 
\begin{cases}  \lambda\prod\limits_{i = 0, i \neq k}^{n-1} \displaystyle\frac{1}{|d_i-d_k|}  & \text{if} \, \alpha = d_i\\
                0                                                                & \text{otherwise}
\end{cases}
\text{\, \, where $\lambda \in \mathbb{Q}_{> 0}$.}
\]
%
We define a partial order on the degree sequences so that ${\bf d}^s < {\bf d}^t$ if $d_i^s \leq d_i^t$ for all $i = 0,1,...,n-1$. 
The order on the degree sequences induces an order of the pure diagrams $\pi_{{\bf d}^s} < \pi_{{\bf d}^t}$ if ${\bf d}^s < {\bf d}^t$.  The BS decomposition of a Betti diagram of an $R$- module is a linear combination of pure diagrams with positive coefficients. 

\begin{algorithm}[{\bf Boij-S\"oderberg Decomposition Algorithm}] \label{BSalgorithm}The algorithm to decompose for a given (non-pure) Betti diagram has the following steps
\begin{enumerate}
\item[(1)] Determine the top degree sequence  ${\bf d}= (d_0,...,d_i,...,d_{n-1})$ of the Betti diagram of the $R$-module $M$, say $\be(M)$.
\item[(2)] Determine the coefficient $\alpha$ of the pure diagram 
$\pi_{{\bf d}}$ by 
\[
\min\left\{\displaystyle\frac{\be_{i,d_i}(M)}{\be_{i,d_i}(\pi_{\bf d})}, \text{ for } i = 0,1,..,n-1\right\}.
\]
\item[(3)] Subtract $\alpha\pi_{{}\bf d}$ from the Betti diagram $\be(M)$ so that the new entries will be all positive.
\item[(4)] Repeat the first and second steps for the remaining diagram $\be(M)- \alpha \pi_{{\bf d}}$ until the Betti diagram is completely decomposed into pure diagrams.
\end{enumerate}
\end{algorithm}
Thus the BS decomposition of a graded $R$-module $M$ gives an ordered decomposition of the Betti diagram,
\[ \be(M) = \sum_{s} a_s \pi_{{\bf d}^s} \text{ where }  \pi_{{\bf d}^s} < \pi_{{\bf d}^t}  \text{ if } s<t. \]
\begin{example}
For instance, let $I = (x^2, xy, xz, y^2)$ be an ideal in $\mathbf{k}[x,y,z]$, the BS decomposition of $R/I$ is given as 
\[
\be(R/I) = (8) \pi_{{\bf d^0}} + (4) \pi_{{\bf d^1}}
\]
where
 \[
 \pi_{{\bf d^0}} =\begin{tabular}{c|cccc}
 $ $ & 0 & 1 & 2 & 3\\
 \hline
 $0$ & $\frac{1}{24}$ & - & - & - \\
 $1$ & - & $\frac{1}{4}$ & $\frac{1}{3}$ & $\frac{1}{8}$\\
 \end{tabular} <
 \pi_{{\bf d^1}}  = \begin{tabular}{c|ccc}
 $ $ & 0 & 1 & 2 \\
 \hline
 $0$ & $\frac{1}{6}$ & - & - \\
 $1$ & - & $\frac{1}{2}$ & $\frac{1}{3}$ \\
 \end{tabular}
 \text{ as }  {\bf d^0} = (0,2,3,4) < {\bf d^1} = (0,2,3).
 \]
 \end{example}
Consider a monomial ideal $I$ in $R$. We will denote the set of minimal monomial generators of $I$ by $G(I)$. Then $G(I)_i $ will denote the subset of $G(I)$ containing the minimal generators of degree $i$. 
The notation $G_{min}(I)$ will be used for the initial degree of the monomials in $I$ and $G_{max}(I)$ will stand for the maximum degree of the monomials in $G(I)$
throughout the paper. We next state the definitions of graded lexicographic monomial order and lex-segment ideal.
\begin{definition}
Let $\mathfrak{m} = x_1^{s_1}...x_n^{s_n}$ and $\mathfrak{n} = x_1^{t_1}...x_n^{t_n}$ be two monomials in $R=\mathbf{k}[x_1,...,x_n]$. 
If either $\deg \mathfrak{m} > \deg \mathfrak{n}$ or $\deg \mathfrak{m} = \deg \mathfrak{n}$
 and $s_i-t_i > 0$ for the first index $i$ such that $s_i-t_t \neq 0$, then it is said that $\mathfrak{m} >_{glex} \mathfrak{n}$ in \textit{graded lexicographic order}.
\end{definition}
\begin{definition}
Let $R$ be a polynomial ring and $L$ be a monomial ideal in $R$ generated by the monomials $m_1,...,m_l$. The ideal $L$
is called a \textit{lex-segment ideal} (lexicographic ideal, or lex ideal) in $R$ if for each monomial $m \in R$ the existence of some $m_i \in G(L)$
with \ $m >_{glex} m_i$ and $\deg(m) = \deg(m_i)$ implies $m\in L$. 
\end{definition}
For simplicity, we will use $>$ for the lex order $>_{glex}$ unless the order is different than lexicographic order. In this section, we make some observations about the Betti diagrams of lex-segment ideals. We aim to get some correlations between Betti numbers of the lex ideals $L = x\mfa +J $, $\mfa = L:x$ and $J$ in $\mathbf{k}[x,y,z]$. Next lemma shows that the colon ideal $\mfa$ is also a lex-segment ideal in $\mathbf{k}[x,y,z]$.
\begin{lemma} Let $L$ be a lex-segment ideal in $R = \mathbf{k}[x_1,..,x_n]$.
Consider the colon ideals $\mfa_i = L : (x_i)$, for $i=1,...,n$. Then each $\mfa_i$ is also a lex-segment ideal in $R$.  
\end{lemma}
\begin{proof}
 Let $m' \in \mfa_i$ be a monomial, for any $i=1,...,n$. Let $m$ be a monomial in $R$ and $\deg m = \deg m'$ and $ m >_{glex} m'$.
Then $x_im' \in L$ as $ \mfa_i = L : (x_i)$, and  $x_im >_{glex} x_im'$. This implies $x_im \in L$ and hence $m \in L : (x_i) = \mfa_i$.
\end{proof}
Let $u$ be a monomial in $R = \mathbf{k}[x_1,..,x_n]$, we define $m(u)$ to be the largest index $i$ such that $x_i$ divides $u$.
Recall that a monomial ideal $I$ is said to be \textit{stable} if, for every monomial $u \in G(I)$ and all $i < m(u)$,  $x_iu/x_{m(u)}$ is also in $G(I)$.

Next we quote a proposition from \cite{EK}.
\begin{proposition}\textbf{(Eliahou-Kervaire formula)}
 Let $I\subset R$ be a stable ideal. Then
\begin{enumerate}
 \item[(a)]$\be_{i,i+j}(I) = \sum\limits_{u\in G(I)_j}\binom {m(u)-1} {i}$,
 \item[(b)]$\projdim{R/I} = \max\{m(u):u\in G(I)\}$,
 \item[(c)]$\reg{(I)} = \max\{\deg{(u)}:u\in G(I)\}$.
\end{enumerate}

\end{proposition}
 
From now on, we always assume $n=3$, that is, $R = \mathbf{k}[x,y,z]$ unless otherwise is stated. We follow with a lemma that indicates the relation between the minimal generators of ideals $L = x\mfa +J $, $\mfa = L:x$ and $J$. Then the next lemma provide a crucial short exact sequence of these ideals.

\begin{lemma}\label{decomp-of-L}
 If $L$ is lex-segment ideal in $R$, then there are unique monomial ideals $\mfa \subset R$ and $J \subset \mathbf{k}[y,z]$ such that $L = x\mfa + J$. Moreover, the ideal $\mfa$ is also a lex-segment ideal since $\mfa = L : (x)$ and $J$ is stable in $R$, and
$G(L) = xG(\mfa) \cup G(J).$ 
\end{lemma}
\begin{proof}
The proof follows immediately from the fact that $L$ is a lex-segment ideal with graded lex order $x>y>z$.
\end{proof}
\begin{lemma}\label{lemma for ses} Let  
$0 \rightarrow F_1 \rightarrow F_0\rightarrow J$  and $0 \rightarrow G_2 \rightarrow G_1 \rightarrow G_0 \rightarrow \mfa$ 
be graded free resolutions for the ideals $J$  and $\mfa$. If $L = x\mfa + J$, then there is a short exact sequence $0 \rightarrow J(-1) \rightarrow \mfa(-1)\oplus J \rightarrow L \rightarrow 0$. 
Moreover, 
$$
0 \rightarrow G_2(-1) \oplus F_1(-1) \rightarrow G_1(-1) \oplus F_1 \oplus F_0(-1) \rightarrow  G_0(-1) \oplus F_0 \rightarrow L
$$
is the graded minimal free resolution of $L$.
\end{lemma}
\begin{proof} The form of the lex-segment ideal $L$ implies the short exact sequence $0 \rightarrow J(-1) \rightarrow \mfa(-1)\oplus J \rightarrow L \rightarrow 0$.
The mapping cone for the short exact sequence provides a free resolution for $L$. 
Let $m\in G(\mfa)\cap G(J)$. Then $m\in G(J)$ implies either $m=ym'$ or $m=zm'$ for some monomial $m' \in \mathbf{k}[y,z]$. As $L$ is a lex-segment
ideal and $m\in L$, $xm' \in G(L)$. So $m' \in G(\mfa)$. Therefore $m$ is divisible by $m'$ and $m$ cannot be a minimal generator of $\mfa$. 
Therefore the ideals $J$ and $\mfa$ do not have common minimal generators. 
This tells us that there is no cancellation in the mapping cone structure. So the resulting graded free resolution for $L$ is minimal.
\end{proof}
We now analyze the Betti numbers of the ideals $L$, $\mfa = L : (x)$ and $J$. We know that the lex-segment ideals $L$ and $\mfa$ are stable and in addition to this,
$J$ is a lex ideal in $\mathbf{k}[y,z]$. Thus,  Eliahou-Kervaire formula gives rise to the following decomposition,
\begin{align*} 
\large \begin{matrix}
 \be_{i,i+j}(L)  =  \sum\limits_{u\in G(L)_j}\binom {m(u)-1} {i} 
                & =&  \underbrace{\sum\limits_{\substack{u\in G(L)_j\\ x|u}}\binom {m(u)-1} {i} }& + &\underbrace{\sum\limits_{\substack{u\in G(L)_j \\ x\nmid u }}\binom {m(u)-1} {i}}\\
                                                               &   & = \be_{i,i+j-1}(\mfa)  &    & \mbox{say} \ D_{i,i+j}
\end{matrix}
\end{align*}                                                             
We name the initial degree of $J$, $G_{min}(J) := k$ and the Betti numbers of $\be(\mfa)$ and $\be(J)$ as 
$$a_{i,i+j}:=\be_{i,i+j}(\mfa),\, \, \mbox{and} \, \,   c_{i,i+j}:=\be_{i,i+j}(J).$$
The following lemmas provide some relations and identities about the Betti numbers of $L$, $\mfa$. They help us to describe the entire
Betti diagram of $L$ with respect to the Betti numbers of $\mfa$ and $J$.
\begin{lemma}\label{initials}
As $L = x\mfa + J$ in $R = \mathbf{k}[x,y,z]$,
if $G_{min}(L) \geq 2$, then $G_{min}(L) = G_{min}(\mfa) + 1$ by stability of the ideals $L$ and $\mfa= L : (x) \neq 1$.\\
\end{lemma} 
\begin{lemma}
%
%
The Eliahou-Kervaire formula gives the following identities for the Betti numbers of the $J$

\begin{itemize}
 \item $c_{0,k} = c_{1,k+1} + 1$,
 \item $c_{0,j} = c_{1,j+1}$ for all $j \geq k+1$,
 \item if $c_{0,k} = k+1$ then $ c_{1,k+1} = k$ and $c_{i,i+j} = 0$ for all $i = 0,1$ and $j\geq k+1$.
\end{itemize}
\vskip .5cm
We know that $\be_{i,i+j}(L) = a_{i,i+j-1} + D_{i,i+j}$. Thus, it follows that
\[D_{i,i+j} = 
\begin{cases}
0, & \text{ when } j\leq k-1,  \\
\be_{i,i+j}(J,x),  &\text{ when }  j \geq k.
\end{cases}\]

That is,
$D_{0,j} = c_{0,j}, D_{1,j+1} = c_{0,j}+c_{1,j+1}, \ \  \mbox{and} \ \ D_{2,j+2} = c_{1,j+1}$.\\
\end{lemma}

\begin{lemma}\label{degree k}
 $G_{min}(J)\geq G_{max}(\mfa)+1$ where $J\neq 0$.
\end{lemma}

\begin{proof}
Say $G_{max}(\mfa)=t$. Suppose $k = G_{min}(J) <t$, then $y^k \in G(L)_k$. So, by lex-order, all monomials $u$ of degree k divisible by $x$ are in $L$. Thus, $u$ is 
in the form $x^iy^jz^s$ where $i\geq 1,\  i+j+s = k$. As $G_{max}(\mfa)=t > k$, there is a minimal generator $v\in L$ of degree $t+1$ such that
$x|v$. Therefore, $v$ can be written as $v = w_1\cdot w_2$ where $w_1$ and $w_2$ are two monomials such that  $\deg w_1 = k$ and $w_2$ is not divisible by $x$ and $\deg w_2 = t-k+1$. Since all degree $k$ monomials divisible by $x$ are in $L$, $w_1$ is in $L$ and so $v$ cannot be a minimal generator.
Thus $k\geq t$.

Now, we need to show that the equality is not possible, i.e. $k=t$ is not possible. We prove this by contradiction. To this end, suppose $k=t$. So $y^k$ is a minimal generator in $L$ and since $t=k$ we can find at least one minimal generator $u$ of $\mfa$ with degree $k$. Then  $xu$  becomes a minimal generator in $L$ of degree $k+1$. As all monomials $v$ of degree $k$  divisible by $x$ are in $L$ then there is a monomial $w$ such that $v=xw$ and $w|u$. This contradicts that $u$ is a minimal generator of $\mfa$. Hence $k\neq t$. i.e. $k\geq t+1$
\end{proof}
Suppose Betti diagrams for $\mfa$ and $J$ are 
\begin{center}
\begin{table}[h]
$ \begin{tabular}{c|ccc}
$\be(\mfa)$ & 0 & 1 & 2\\
\hline
$1$ & $a_{0,1}$ & $a_{1,2}$ & $a_{2,3}$\\
$2$ & $a_{0,2}$ & $a_{1,2}$ & $a_{2,4}$\\
\vdots & \vdots & \vdots & \vdots \\
$k-1$ & $a_{0,k-1}$ & $a_{1,k}$ & $a_{2,k+1}$

\end{tabular}$ \ \ \   and  \ \ \ $ \begin{tabular}{c|cc}
$\be(J)$ & 0 & 1 \\
\hline
$k$ & $c_{0,k}$ & $c_{0,k}-1$\\
$k+1$ & $c_{0,k+1}$ & $c_{0,k+1}$\\
\vdots & \vdots & \vdots \\
$G_{max}(J)$ & $c_{0,G_{max}(J)}$ & $c_{0,G_{max}(J)}$
\end{tabular}$,
\vspace*{0.3 cm}
\caption{The Betti diagrams of $\mfa$ and $J$.}
\end{table}
\end{center}
Therefore the short exact sequence in Lemma \ref{lemma for ses} together with all other Lemmas \ref{decomp-of-L}, \ref{initials} and \ref{degree k} we discuss in this section yield that the Betti diagram for $L$ has the following form: 
\begin{center}
\begin{table}[h]
\begin{tabular}{c|ccc}
$\be(L)$ & 0 & 1 & 2\\
\hline
$2$ & $a_{0,1}$ & $a_{1,2}$ & $a_{2,3}$\\
$3$ & $a_{0,2}$ & $a_{1,2}$ & $a_{2,4}$\\
\vdots & \vdots & \vdots & \vdots \\
$k-1$ & $a_{0,k-2} $ & $a_{1,k-1}$ & $a_{2,k}$\\
$k$ & $a_{0,k-1}+c_{0,k} $ & $a_{1,k}+2c_{0,k}-1$ & $a_{2,k+1}+c_{0,k}-1$\\
$k+1$ & $c_{0,k+1}$ & $2c_{0,k+1}$ & $c_{0,k+1}$\\
\vdots & \vdots & \vdots & \vdots \\
$G_{max}(L) = G_{max}(J)$ & $c_{0,G_{max}(L)}$ & $2c_{0,G_{max}(L)}$ & $c_{0,G_{max}(L)}$
\end{tabular}
\vspace*{0.3 cm}
\caption{The Betti diagram of $L$}
\label{Betti L} 
\end{table}
\end{center}
If the Betti diagrams of the ideals $\mfa$ and $J$ ``overlap" then they do only at the $k^{th}$ row of the $\be(L)$ as in the above diagram. In other words, the Betti numbers of $\be(L)$ in the $k^{th}$ row may be expressed in terms of both the Betti numbers of $\be(\mfa)$ and $\be(J)$ in their $k-1^{th}$ and first row, respectively.


\section{The Boij-S\"{o}derberg Decompositions of $L$ and $L :(x)$}
\label{thm1section}


In this section we prove Theorem \ref{thm1}. Let
$\bf{d^0} < \bf{d^1} < ...<\bf{d^i}<... < \bf{d^t}$ be the BS chain of all length $3$ top degree sequences for $\mfa = L : (x)$.
Suppose the chain of the first $t+1$ top degree sequences of the BS decomposition of the Betti diagram of $L$ is 
$\bf{\bar{d}^0} < \bf{\bar{d}^1} < ...< \bf{\bar{d}^i}<...<\bf{\bar{d}^t}$. Then 
${\bf \bar{d}^i} = {\bf d^i} + {\bf1} = (d^i_0+1, d^i_1+1, d^i_2+1 )$ for all $i = 0,1,...,t$ with 
exactly the same coefficients, except possibly the coefficient of $\pi_{{\bf \bar{d}^t}}$. 

Recall that, for a given top degree sequence ${\bf d} = (d_0, d_1, d_2)$, the ``normalized'' pure diagram $\pi_{\bf{d}}$ can be obtained as
\[
\be_{i,i+j}(\pi_{\bf{d}}) = 
\begin{cases}
  0 & \text{if\ } i+j\neq  d_i \\
  \lambda\prod\limits_{r=0, r\neq i}^{2}\frac{1}{|d_i-d_r|} & \text{if\ } i+j=d_i,\ \ \mbox{where\ } \lambda = \lcm\left(\prod\limits_{r=0, r\neq i}^{2} |d_i-d_r|, i = 0,1,2\right).
 \end{cases}
 \]
Thus, this formula provides pure diagrams with integer entries. From now on,  pure diagrams have integer entries. Let ${\bf d^0} = (d^0_0, d^0_1, d^0_2)$ be the top degree sequence for the Betti diagram of $\mfa$.
if $d^0_2 < k+1$, that is, $d^0_0<d^0_1<d^0_2<k+1$,  so $d^0_0 < k-1$ then we see that $\be_{i,i+j}(\mfa) = \be_{i,i+j+1}(L)$ for all $j = 0,1,...,k-2$ as in the table \ref{Betti L}.
This essentially follows from the fact that the Betti diagrams of $\mfa$ and $J$ may only overlap on the $k$-th row in the Betti diagram of $L$. As $L = x\mfa + J$ and degree shift
due to multiplication by $x$ the top degree sequence of $\be(L)$ will be ${\bf d^0} + {1}$. Thus $\be(L)-\al_0\pi_{d^0+1}$ 
becomes the first step of the BS-decomposition of $\be(L)$. In fact, we could repeat this process for all degree sequence ${\bf d^s}$ such that $d_2^s < k+1$.

Suppose $d^s_2 < k+1$ for $s=0, 1, ..., l-1$ and $d^s_2 \leq k+1$ for $s \geq l$ for some $l > 0$. So we assume that the next degree sequence after ${\bf d^{l-1}}$ is ${\bf d^l} = (d^l_0, d^l_1, d^l_2 = k+1)$. Therefore, after $l$ steps 
in the decomposing both $\be(\mfa)$ and $\be(L)$, we would get the remaining diagrams  
$$\be(\mfa)-\sum\limits_{s=0}^{l-1}\al_s\pi_{{\bf d^s}} =: \tilde{\be}(\mfa) \ \ \ \ \mbox{and} \ \ \ \ \be(L)-\sum\limits_{s=0}^{l-1}\al_s\pi_{{\bf d^s+1}} =: \tilde{\be}(L).$$ 
Let ${\bf d^l} = (d^l_0, d^l_1, d^l_2) $ be the next top degree sequence of the Betti diagram for $\mfa$ and $d^l_2=k+1$ so above paragraph shows that 
${\bf d^l+1}$ becomes the next top degree sequence of Betti diagram for $L$.
Then the remaining diagram after the first $l$ steps of the BS
decompositions for both $\mfa$ and $L$ look like as following,
\begin{center}
\begin{table}[h]
$\be(\mfa)-\sum\limits_{s=0}^{l-1}\al_s\pi_{{\bf d^s}} =\small\begin{tabular}{c|ccc}
$\tilde{\be}(\mfa)$ & 0 & 1 & 2\\
\hline
$d^l_0$ & $\tilde{\be}_{0,d^l_0}(\mfa)$ & - & -\\
\vdots & \vdots & \vdots & \vdots \\
$d^l_1-1$ & $a_{0,d^l_1-1}$ & $\tilde{\be}_{1,d^l_1}(\mfa)$ & -\\
\vdots & \vdots & \vdots & \vdots \\
$d^l_2-2=k-1$ & $a_{0,d^l_2-2}$ & $a_{1,d^l_2-1}$ & $a_{2,d^l_2}$
\end{tabular}$
\\
\caption{Remaining diagram after $l$ steps for $\be(\mfa)$}
\label{remaining a}
\end{table}
\end{center}
and similarly, 
\begin{center}
\begin{table}[h]
$\be(L)-\sum\limits_{s=0}^{l-1}\al_s\pi_{{\bf d^s+1}} = \small\begin{tabular}{c|ccc}
$\tilde{\be}(L)$ & 0 & 1 & 2\\
\hline
$d^l_0+1$ & $\tilde{\be}_{0,d^l_0+1}(L) $ & - & -\\
\vdots & \vdots & \vdots & \vdots \\
$d^l_1$ & $a_{0,d^l_1-1}$ & $\tilde{\be}_{1,d^l_1+1}(L)$ & -\\
\vdots & \vdots & \vdots & \vdots \\
$d^l_2-1=k$ & $a_{0,d^l_2-2}+c_{0,d^l_2-1}$ & $a_{1,d^l_2-1}+c_{0,d^l_2-1}+c_{1,d^l_2}$ & $\tilde{\be}_{2,d^l_2+1}(L) $\\
$d^l_2$  & $c_{0,d^l_2}$ & $c_{0,d^l_2}+c_{1,d^l_2+1}$ & $c_{1,d^l_2+1}$\\
\vdots & \vdots & \vdots & \vdots 
\end{tabular}$
\\
\caption{Remaining diagram after $l$ step for $\be(L)$.}
\label{remaining L}
\end{table}
\end{center}
By construction of $\be(L)$, we deduce that
\begin{align*}
\tilde{\be}_{0,d_{0}^l+1}(L) &= \tilde{\be}_{0,d_{0}^l}(\mfa) \, \mbox{ as } \,  d_0^l+1 < k, \\
\tilde{\be}_{1,d_{1}^l+1}(L) &= \tilde{\be}_{1,d_{1}^l}(\mfa) \,  \mbox{ as } \, d_1^l < k, \, \mbox{ and } \\ 
\tilde{\be}_{2,d_{2}^l+1}(L) &= a_{2,d^l_2} + c_{1,d_2^l} \, \mbox{ as } \,  d_2^l-1 = k.
\end{align*}
The algorithm \ref{BSalgorithm} exposes the coefficient of the pure diagram $\pi_{{\bf d^l}}$ to be
\begin{eqnarray}\label{l-coefficient of a}
\al_l =\min\left\{ \displaystyle\frac{\tilde{\be}_{0,d_0^l}(\mfa)}{\be_{0,d_0^l}(\pi_{{\bf d^l}})}, 
                                   \displaystyle\frac{\tilde{\be}_{1,d_1^l}(\mfa)}{\be_{1,d_1^l}(\pi_{{\bf d^l}})}, 
                                   \displaystyle\frac{a_{2,d_2^l}}{\be_{2,d_2^l}(\pi_{{\bf d^l}})}\right\}.
\end{eqnarray}
and similarly for the BS-decomposition of $\be(L)$ there is a rational number
$\tilde{\al_l}$  as the coefficient of the pure diagram $\pi_{{\bf d^l +1}}$ such that

 \begin{eqnarray}\label{l-coefficient of L}
\tilde{\al_l} = \min \left\{ \displaystyle\frac{\tilde{\be}_{0,d_0^l}(\mfa)}{\be_{0,d_0^l}(\pi_{{\bf d^l}})}, 
                                   \displaystyle\frac{\tilde{\be}_{1,d_1^l}(\mfa)}{\be_{1,d_1^l}(\pi_{{\bf d^l}})}, 
                                   \displaystyle\frac{a_{2,d_2^l} + c_{1,d^l_2}}{\be_{2,d_2^l}(\pi_{{\bf d^l}})}\right\}
\end{eqnarray}
Hence we just need to look at the $k$-th row of the Betti diagram of $L$ if $\be(\mfa)$ and $\be(L)$ overlap. Thus, we only need 
to think about the top degree sequences $d^s$ of length $3$ of $\be(\mfa)$ such that $d^s_2 = k+1$.\\
\textit{\textbf{Case 1}}:
Let $a_{2,k+1}$ be eliminated in the $(l+1)$-th step of the decomposition algorithm of $\be(\mfa)$. In other words,
${\bf d^l} = (d_0^{l},d_1^{l}, d_2^{l})$ is of length $3$, whereas ${\bf d^{l+1}} = (d_0^{l+1},d_1^{l+1})$
has length $2$. It shows that $\bf{d^0} < \bf{d^1} < ...<\bf{d^i}<... < \bf{d^{l}}$ are all length $3$ degree sequences in the decomposition of $\be(\mfa)$.
Hence, BS decomposition of $\be(\mfa)$ is 
$$\be(\mfa) = \sum\limits_{s=0}^{l}\al_s\pi_{{\bf d^s}} + [\mbox{all pure diagrams of length less than $3$}].$$
Recall that we only focus on the degree sequences of length three. Since the length of $d^{l+1}$ is two, we do not need to pay attention to the $(l+2)$-th step in the decomposition.
Besides that the table \ref{remaining L} already shows that ${\bf d^l + 1}$ is top degree sequence of the remaining diagram of $L$, $\tilde{\be}(L)$.
Therefore, the first $(l+1)$-th top degree sequences of Boij-S\"{o}derberg decomposition
of $\be(L)$ is $${\bf d^0+1}<{\bf d^1+1}<...<{\bf d^{l}+1}$$ 
where the coefficients $\tilde{\al}_i = \al_i$ for $i = 0,1,...,l-1$.\\
\textit{\textbf{Case 2}}:
Suppose that $a_{2,k+1}$ is not eliminated in the $(l+1)$-th step of the 
decomposition of 
$\be(\mfa)$. Moreover we assume that it will vanish in the $(t+1)$-th step for 
some $t>l$. That is, the chain of the degree sequences in the BS decomposition 
of $\be(\mfa)$ is
$$\bf d^0< d^1< ...< d^l< ...<d^t<...<d^n$$ where
\begin{itemize} 
\item for $s=0,1,...,l-1$, \ ${\bf d^s}=(d_0^s, d_1^s, d_2^s)$ has length $3$ such that $d^s_2< k+1$,
\item for $s=l,...,t$, \ ${\bf d^s}=(d_0^s, d_1^s, d_2^s)$ has length $3$ such that $d^s_2=k+1$,
\item for $s=t
+1,...,n$, \ ${\bf d^s}= (d^s_0, d^s_1)$ has length $2$. 
\end{itemize}
As the entries only above the $(k-1)$-th row are eliminated until the $(t+1)$-th step of the decomposition, it is easy to observe the remaining diagram of $L$.  In table \ref{Betti L}, we have seen that the entries of both $\be(\mfa)$ and $\be(L)$ above the $k$-th row are the same. Therefore, the remaining diagram of $\be(\mfa)$ after subtracting the first $t$ pure diagrams is
\begin{center}$\be(\mfa)-\sum\limits_{s=0}^{t-1}\al_s\pi_{{\bf d^s}} =\small\begin{tabular}{c|ccc}
$\tilde{\be}(\mfa)$ & 0 & 1 & 2\\
\hline
$d^t_0$ & $\tilde{\be}_{0,d^t_0}(\mfa)$ & - & -\\
\vdots & \vdots & \vdots & \vdots \\
$d^t_1-1$ & $a_{0,d^t_1-1}$ & $\tilde{\be}_{1,d^t_1}(\mfa)$ & -\\
\vdots & \vdots & \vdots & \vdots \\
$d^t_2-2=k-1$ & $a_{0,d^t_2-2}$ & $a_{1,d^t_2-1}$ & $\tilde{\be}_{2,d^t_2}(\mfa)$
\end{tabular}$
\end{center} 
where $ \tilde{\be}_{i,d^t_i}(\mfa) = \be_{i,d^t_i}(\mfa)-\sum\limits_{s=0}^{t-1}\al_s\be_{i,d^t_i}(\pi_{{\bf d^s}})$, for $i=0,1,2$. Furthermore, as in (\ref{l-coefficient of a}) and (\ref{l-coefficient of L}), we have similar relations between 
the coefficients in both BS decomposition of $\be(\mfa)$ and $\be(L)$ during their first $t$ steps.
The coefficients of the pure diagrams $\pi_{{\bf d^s}}$ in the decomposition of $\be(\mfa)$  for $s = l,...,t-1$ is
$$\al_s = \min\left\{ \displaystyle\frac{\tilde{\be}_{i,d_i^s}(\mfa)}{\be_{i,d_i^s}(\pi_{{\bf d^s}})}, \ \ \mbox{for} \ \ i=0,1,2\right\}.$$
Similarly, the corresponding coefficient $\tilde{\al_s}$ of the pure diagram $\pi_{{\bf d^s +1}}$ in the decomposition of $\be(L)$ becomes
\begin{eqnarray*}\tilde{\al_s} &=&\min\left\{ \displaystyle\frac{\tilde{\be}_{i,d_i^s+1}(L)}{\be_{i,d_i^s+1}(\pi_{{\bf d^s+1}})},\ \  \mbox{for} \ \  i=0,1,2\right\}\\
                                &= &\min\left\{ \displaystyle\frac{\tilde{\be}_{0,d_0^s}(\mfa)}{\be_{0,d_0^s}(\pi_{{\bf d^s}})}, 
                                   \displaystyle\frac{\tilde{\be}_{1,d_1^s}(\mfa)}{\be_{1,d_1^s}(\pi_{{\bf d^s}})}, 
                                   \displaystyle\frac{\tilde{\be}_{2,d_2^s}(\mfa) + c_{1,d^s_2}}{\be_{2,d_2^s}(\pi_{{\bf d^s}})}\right\}.
 \end{eqnarray*}
 We assume that any of the entries of the corresponding to $d^s_i$ for $i=0,1$ is eliminated where $s = l,..,t-1$. Thus
 $$\al_s < \frac{\tilde{\be}_{2,d_2^s}(\mfa)}{\be_{2,k+1}(\pi_{d^s})}, \ \ \ \mbox{where} \ \ \ d_2^s = k+1.$$
 So it follows that 
 $$\displaystyle\frac{\tilde{\be}_{2,d_2^s}(\mfa)}{\be_{2,k+1}(\pi_{{\bf d^s}})} < \displaystyle\frac{\tilde{\be}_{2,d_2^s}(\mfa)+c_{1,k+1}}{\be_{2,k+1}(\pi_{{\bf d^s}})}.$$
 Hence $\tilde{\al}_s=\al_s$ for $s=l,...,t-1$. However, this equality may not be true for the coefficients $\al_t$ and $\tilde{\al}_t$ since $\tilde{\be}_{2,d^t_2}(\mfa)$ will be
eliminated in the next step. So $\al_t \leq \tilde{\al}_t$.
Hence the remaining diagram of $\be(L)$ is \\
$\tilde{\be}(L) :=\be(L)-\sum\limits_{s=0}^{t-1}\al_s\pi_{{\bf d^s+1}} = \small
\begin{tabular}{c|ccc}
$ $ & 0 & 1 & 2\\
\hline
$d^t_0+1$ & $\tilde{\be}_{0,d^t_0+1}(L) $ & - & -\\
\vdots & \vdots & \vdots & \vdots \\
$d^t_1$ & $a_{0,d^t_1-1}$ & $\tilde{\be}_{1,d^t_1+1}(L)$ & -\\
\vdots & \vdots & \vdots & \vdots \\
$d^t_2-1=k$ & $a_{0,d^t_2-2}+c_{0,d^t_2-1}$ & $a_{1,d^t_2-1}+c_{0,d^t_2-1}+c_{1,d^t_2}$ & $\tilde{\be}_{2,d^t_2}(L)$ \\
$k+1$ & $c_{0,k+1}$ & $c_{0,k+1}+c_{1,k+2}$ & $c_{1,k+2}$ \\
\vdots & \vdots & \vdots & \vdots \\
\end{tabular}$
where 
\begin{align*}
\tilde{\be}_{0,d_{0}^t+1}(L) &= \tilde{\be}_{0,d_{0}^t}(\mfa) \, \mbox{ as } \,  d_0^t+1 < k, \\
\tilde{\be}_{1,d_{1}^t+1}(L)& = \tilde{\be}_{1,d_{1}^t}(\mfa) 
\, \mbox{ as } \, d_1^t < k, \, \mbox{ and } \\ 
\be_{2,d_{2}^t+1}(L) &= \tilde{\be}_{2,d^t_2}(\mfa) + c_{1,d_2^t} \ \ \mbox{as} \ \ d_2^t-1 = k.
\end{align*} 
This will bring us back to {\bf Case 1}; ${\bf d^t} = (d^t_0,d^t_1,d^t_2)$ is the last top degree sequence of length $3$ in the BS decomposition of $\be(\mfa)$.
The remaining diagram above clearly shows us that ${\bf d^t+1} = (d^t_0+1,d^t_1+1,d^t_2+1)$ shows up as a degree sequence in the BS decomposition of $\be(L)$ in the next step.

As a summary, if $\bf d^0<d^1<...<d^t$ is the chain of the all top degree sequences of length $3$ in the BS-decomposition of 
$\be(\mfa)$ with coefficients $\al_s$ for $s=0,1,..,t$.
Then
$\bf d^0+1<d^1+1<...<d^t+1$ becomes the first $t$ top degree sequences of length $3$ in the BS-decomposition of $\be(L)$ with
$\tilde{\al}_s= \al_s \ \ \mbox{if}\ \ s<t$ and $\tilde{\al}_t \geq \al_t$. 

Hence we have shown that the "beginnings`` of the chain of the degree sequences in the BS decompositions of $\be(L) = x\mfa + J$ and $\be(\mfa)$ are identical. We believe that there is a analogous result for a lex-ideal in $\mathbf{k}[x_1, ..., x_n]$.
\begin{remark} Let $L= (x_1)\mfa+J$ in $R=\mathbf{k}[x_1,x_2,...,x_n]$ be a lex-segment ideal, then $\mfa$ is also lex-segment ideal in $R$ and $J$ turns out to be a stable ideal of $\codim n-1$ in $\mathbf{k}[x_2,...,x_n]$. 

Suppose 

\begin{eqnarray*}\minCDarrowwidth 5pt \begin{CD}
F_{n-1} @>>> \cdots  @>>> F_i @>>> \cdots @>>> F_1 @>>> J @>>> 0    
                \end{CD}
\end{eqnarray*}
\begin{eqnarray*}\minCDarrowwidth 5pt \begin{CD}
G_n @>>> \cdots  @>>> G_i @>>> \cdots @>>> G_1 @>>> \mfa @>>> 0    
                \end{CD}
\end{eqnarray*}
 
are the minimal free resolutions of $J$ and $\mfa$, respectively.  We get the same short exact sequence in Lemma \ref{lemma for ses}, 
then by mapping cone we have the following minimal free resolution for $L$
 \begin{center}
 $0 \rightarrow G_n(-1) \oplus F_{n-1}(-1) \rightarrow ... \rightarrow  G_2(-1) \oplus F_2 \oplus F_1(-1) \rightarrow  G_1(-1) \oplus F_1 \rightarrow L$. 
 \end{center} 
 
So it yields 

$$\be_{i,i+j}(L) = \begin{cases}
                      \be_{i,i+j-1}(\mfa)  &\mbox{where} \ \ i=0,1,...,n-1 \ \ \mbox{and} \ \ i+j< G_{min}(J),\\
                      \be_{i,i+j-1}(\mfa) + \sum\limits_{t=i-1}^i \be_{t,j+t}(J) &\mbox{where}  \ i=0,1,...,n-1 \ \ \mbox{and} \ \ i+j\geq G_{min}(J).
                     \end{cases}$$ 

By using lex-order properties of $L$ and $a$, as we did in case $n=3$, we conclude that the Betti diagrams of $\mfa$ and $J$ either overlap only on the $G_{min}(J)$-th row of the Betti diagram of $L$ or
do not overlap at all. Identify $k:= G_{min}(J)$. Therefore, the Betti diagram of $L$ in $\mathbf{k}[x_1, ..., x_n]$ is

\begin{table}[h]
\begin{center}
\resizebox{\columnwidth}{!}{
\begin{tabular}{c|ccccc}
$\be(L)$ & 0 & 1 & 2 & ... & n-1\\
\hline
$2$ & $a_{0,1}$ & $a_{1,2}$ & $a_{2,3}$ & ... & $a_{n-1,n}$\\
$3$ & $a_{0,2}$ & $a_{1,3}$ & $a_{2,4}$ & ... & $a_{n-1,n+1}$\\
\vdots & \vdots & \vdots & \vdots & ... & \vdots \\
$k-1$ & $a_{0,k-2} $ & $a_{1,k-1}$ & $a_{2,k}$ & ... & $a_{n-1,k+n-3}$\\
$k$ & $a_{0,k-1}+c_{0,k} $ & $a_{1,k}+c_{0,k}+c_{1,k+1}$ & $a_{2,k+1}+c_{1,k+1}+c_{2, k+2}$ & ... & $a_{n-1,k+n-2}+c_{n-1,k+n-1}$\\
$k+1$ & $c_{0,k+1}$ & $c_{0,k+1} + c_{1,k+2}$ & $c_{1,k+2} + c_{2,k+3}$ & ... & $c_{n-1,k+n-1}$\\
\vdots & \vdots & \vdots & \vdots & ... & \vdots \\
$G_{max}(L) = G_{max}(J)$ & $c_{0,G_{max}(L)}$ & $c_{0,G_{max}(L)} + c_{1,G_{max}(L)+1}$ & $c_{1,G_{max}(L)+1} + c_{2,G_{max}(L)+2}$ & ... & $c_{n-1,G_{max}(L)+n-1}$
\end{tabular}
}
\vspace*{0.3 cm}
\caption{Betti diagram of $L$ in $\mathbf{k}[x_1,...,x_n]$}
\end{center}
\end{table}
Henceforth, proof of Theorem \ref{thm1} can be easily modified for the polynomial ring of $n$ variables.
\end{remark}                     
\begin{corollary} Let $L= (x_1)\mfa+J$ in $R=\mathbf{k}[x_1,x_2,...,x_n]$ be a lex-segment ideal. If 
$\pi_{{\bf d}^0} < \pi_{{\bf d}^1} < ... < \pi_{{\bf d}^t}$ 
are all pure diagrams of length $n$ in the BS decomposition of $\mfa$, where  
${\bf d}^i = (d_0^i, d_1^i,..., d_{n-1}^i)$ for $i = 0,1,...,t$. Then the chain of pure diagrams
\[\pi_{{\bf \bar{d}^0}} < \pi_{{\bf \bar{d}^1}} < ... <\pi_{{\bf \bar{d}^t}}\]
appears in the beginning of the BS decomposition of $L$ such that
\[{\bf \bar{d}^i} = {\bf d^i + 1} = (d_0^i +1, d_1^i +1,.., d_{n-1}^i +1).\] 
\end{corollary}

\section{The Boij-S\"oderberg Decomposition for  $L$ and $(L,x)$}
\label{thm2newsection}

In Theorem \ref{thm1}, we have showed that if $\pi_{{\bf d^i}}$, a type-${\bf d^i} = (d^i_0, d^i_1, d^i_2)$ pure diagram, appears as a summand in the BS decomposition of $\be(\mfa)$ then, $\pi_{\bf (d+1)^i}$ for ${\bf (d+1)^i} = (d^i_0+1, d^i_1+1, d^i_2+1)$ show up as the $i$-th summand of the BS decomposition of $\be(L)$,  with the same BS coefficient possible except for the last one. In this section, we now consider the end of the BS decomposition of $L$ in $R=\mathbf{k}[x,y,z]$ and show that all degree sequences of length less than $3$ in the decomposition of $\be(L,x) = \be(J,x) $ occurs precisely as all degree sequences of length less that $3$ in the decomposition for $L$. 
We prove our claim for all Artinian lex-segment ideals $L = \mfa(x)+J$ except the ones of the form $ L = x(x,y,z^t) + J$ where $J$ is different that $(y,z)^{G_{min}(J)}$ and $1 < t < k-1$. The main idea of proof is induction whose base step also requires some tedious case analyzing of the decompositions of both $\be(L)$ and $\be(L,x)$. Then finally schemes of case analyzing help us to demonstrate how all degree sequences of length less than $3$  for both $L$ and $(L,x)$ coincide entirely with their coefficients.

Furthermore, we conjecture that the statement of the Theorem \ref{thm2} is also true for  $L= x(x,y,z^t) + J$, whereas proof of that situation requires a case analyzing which becomes infeasible.  

\subsection{Decomposing the Betti diagram of ${\mathbf (L,x)}$}

First we observe that the same pure diagrams $\pi_{{\bf d}}$, for all ${\bf d}$ of length less than $3$, as in the decomposition of the Betti diagram of $L$. To  show this,  it suffices to check on the remaining diagrams after several steps of the decomposition algorithm for $(L,x)$.  We also notice that, for all $i > k$, $i$-th row of the Betti diagram of $(L,x)$ has the form $| c_{0,i} \, \ , 2c_{0,i+1} \, \, c_{0,i} |$ where $c_{0,i} < k$. 

Say $G_{max}(L,x)= G_{max}(L)=G_{max}(J)= : n$. 
Assume $k=G_{min}(J)>2$ and $n\geq k+1$. Then the Betti diagram of the lex ideal $(L,x)$  is

\begin{center}$\be(L,x) =\small \begin{tabular}{c|ccc}
$ $ & 0 & 1 & 2\\
\hline
$1$ & \boxed{$1$} & $-$ & $-$\\
$2$ & $-$ & $-$ & $-$\\
\vdots & \vdots & \vdots & \vdots \\
$k-1$ & $-$ & $-$ & $-$\\
$k$ & $c_{0,k} $ & $\boxed{2c_{0,k}-1}$ & $\boxed {c_{0,k}-1}$\\
$k+1$ & $c_{0,k+1}$ & $2c_{0,k+1}$ & $c_{0,k+1}$\\
\vdots & \vdots & \vdots & \vdots \\
$n$ & $c_{0,n}$ & $2c_{0,n}$ & $c_{0,n}$\\
\end{tabular}$
\end{center}


First degree sequence is ${\bf \bar{d}^0}=(1,k+1,k+2)$, then we have $\be(L,x) - \ga_0\pi_{{\bf \bar{d}^0}}$ where 
\[\pi_{{\bf \bar{d}^0}}=\tiny \begin{tabular}{c|ccc}
$ $ & 0 & 1 & 2\\
\hline
$1$ & $\bf 1$ & $-$ & $-$\\
\vdots & \vdots & \vdots & \vdots \\
$k$ & $- $ & $\bf k+1$ & $\bf k$
\end{tabular} \, \, \, \text{ and } \, \, \,
 \ga_{0}=\min\left\{1,\displaystyle\frac{2c_{0,k}-1}{k+1}, \displaystyle\frac{c_{0,k}-1}{k}\right\} = \displaystyle\frac{c_{0,k}-1}{k}.\]

Notice that Artinian lex ideal property of $L$ yields $c_{0,k} \leq k+1$. Thus $\frac{c_{0,k}-1}{k} \leq 1$.
 
If $c_{0,k}=k+1$, i.e. $G_{max}(L)=k$, then the BS decomposition of $(L,x)$ becomes
 $$\be(L,x) = (1)\pi_{(1,k+1,k+2)}+(k)\pi_{(k,k+1)}+(1)\pi_{(k)}.$$

If $c_{0,k} < k+1$, then $\frac{c_{0,k}-1}{k} <1$. Therefore,
the remaining diagram after the first step becomes
 
 \begin{center}$\be(L,x)- \ga_{0}\pi_{{\bf \bar{d}^0}} = \small \begin{tabular}{c|ccc}
$ $ & 0 & 1 & 2\\
\hline
$1$ & $\boxed{1-\frac{c_{0,k}-1}{k}}$ & - &  - \\
\vdots & \vdots & \vdots & \vdots \\
$k$ & $c_{0,k} $ & $\boxed{\frac{(k-1)c_{0,k}+1}{k}}$  & -  \\
$k+1$ & $c_{0,k+1}$ & $2c_{0,k+1}$ &  $\boxed{c_{0,k+1} }$ \\
\vdots & \vdots & \vdots & \vdots \\
$n$ & $c_{0,n}$ & $2c_{0,n}$ & $c_{0,n}$\\
\end{tabular}$
\end{center}

Then the next pure diagram is
\[\pi_{{\bf \bar{d}^1}}=\tiny \begin{tabular}{c|ccc}
$ $ & 0 & 1 & 2\\
\hline
$1$ & $\bf 2$ & $-$ & $-$\\
\vdots & \vdots & \vdots & \vdots \\
$k$ & $- $ & $\bf k+2$ & $-$\\
$k+1$ & $- $ & $-$ & $\bf k$
\end{tabular} \, \, \, \,\, \text{ for } \, \, \, \, {\bf \bar{d}^1}=(1,k+1,k+3).\]

The coefficient for $\pi_{{\bf \bar{d}^1}}$ is
\[\ga_{1}= \min\left\{\frac{1}{2}-\frac{c_{0,k}-1}{2k}, \frac{c_{0,k}(k-1)+1}{k(k+2)}, \frac{c_{0,k+1}}{k}\right\}.\]


{\bf Case I:} Let \[\ga_1 = \displaystyle{\frac{1}{2}-\frac{c_{0,k}-1}{2k}}.\] This is implied by the following inequalities;
\[\frac{k}{3}+1 \leq c_{0,k} \text{ and } 2c_{0,k+1}+c_{0,k} > k+1.\]
 
Thus the algorithm eliminates the entry $\be_{0,1}(L,x)$and the remaining diagram is $$\be(L,x)- \left(\displaystyle\frac{c_{0,k}-1}{k}\right)\pi_{(1,k+1,k+2)}-\left(\displaystyle{\frac{1}{2}-\frac{c_{0,k}-1}{2k}}\right)\pi_{(1,k+1,k+3)} \, \, $$
\begin{center}$ ={\small \begin{tabular}{c|ccc}
$ $ & 0 & 1 & 2\\
\hline
$k$ & $\boxed{c_{0,k}} $ & $\boxed{\frac{3c_{0,k}-k-3}{2}}$ & $-$\\
$k+1$ & $c_{0,k+1}$ & $2c_{0,k+1}$ & $\boxed{c_{0,k+1}-k(\frac{1}{2}-\frac{c_{0,k}-1}{2k})}$\\
\vdots & \vdots & \vdots & \vdots \\
$n$ & $c_{0,n}$ & $2c_{0,n}$ & $c_{0,n}$\\
\end{tabular}}$
\end{center}

Next if ${\bf \bar{d}^2} = (k, k+1, k+3)$ and 
then the coefficient of the pure diagram $\pi_{(k, k+1, k+3)}$ becomes 
\[
\ga_2 = \min\left\{\frac{c_{0,k}}{2}, \frac{3c_{0,k}-k-3}{6} , c_{0,k+1}-\frac{k+1 -c_{0,k}}{2}\right\}.
\] 
This creates two possible sub-cases and we observe the remaining diagrams for each case;

%
%
\underline{Case I.1:} If $\ga_2 =\displaystyle\frac{3(c_{0,k-1})-k}{6}$
which is a result of  $ \frac{k}{3} < c_{0,k+1}, $ then  we obtain

\begin{center}$\be(L,x)- \ga_0\pi_{{\bf \bar{d}^0}}-\ga_1\pi_{{\bf \bar{d}^1}}-\ga_2\pi_{{\bf \bar{d}^2}}= \small \begin{tabular}{c|ccc}
$ $ & 0 & 1 & 2\\
\hline
$k$ & $\frac{k}{3}+1$ & - & -\\
$k+1$ & $c_{0,k+1}$ & $2c_{0,k+1}$ & $c_{0,k+1}-\frac{k}{3}$ \\
\vdots & \vdots & \vdots & \vdots \\
$n$ & $c_{0,n}$ & $2c_{0,n}$ & $c_{0,n}$\\
\end{tabular}$
\end{center}
%
\underline{Case I.2:} If $\ga_2 = c_{0,k+1}-\frac{k+1 -c_{0,k}}{2}$ as $c_{0,k+1} < \frac{k}{3}$ , then 

$$\be(L,x)- \ga_0\pi_{{\bf \bar{d}^0}}-\ga_1\pi_{{\bf \bar{d}^1}}-\ga_2\pi_{{\bf \bar{d}^2}} = \small \begin{tabular}{c|ccc}
$ $ & 0 & 1 & 2\\
\hline
$k$ & $k+1-2c_{0,k} $ & $k-3c_{0,k+1}$ & -  \\
$k+1$ & $c_{0,k+1}$ & $2c_{0,k+1}$ &  -  \\
$k+2$ & $c_{0,k+2}$ & $2c_{0,k+2}$ &  $c_{0,k+2}$  \\
\vdots & \vdots & \vdots & \vdots \\
$n$ & $c_{0,n}$ & $2c_{0,n}$ & $c_{0,n}$\\
\end{tabular}$$



{\bf Case II:} Let 
\[
\ga_1 = \displaystyle{\frac{c_{0,k}(k-1)+1}{k(k+2)}},
\]
because 
\[\frac{k}{3} +1 > c_{0,k} \text{ and } (k-1)c_{0,k}+1 < (k+2)c_{0,k+1}.\] Then the algorithm gives 
$$\be(L,x)- \left(\displaystyle\frac{c_{0,k}-1}{k}\right)\pi_{(1,k+1,k+2)}-\left(\displaystyle{\frac{c_{0,k}(k-1)+1}{k(k+2)}}\right)\pi_{(1,k+1,k+3)}$$

\begin{center}$  =\small \begin{tabular}{c|ccc}
$ $ & 0 & 1 & 2\\
\hline
$1$ & $\boxed{1-\frac{k-3(c_{0,k}-1)}{k+2}}$ & - &  - \\
\vdots & \vdots & \vdots & \vdots \\
$k$ & $c_{0,k} $ & -  & -  \\
$k+1$ & $c_{0,k+1}$ & $\boxed{2c_{0,k+1}}$ &  $\boxed{c_{0,k+1} - \frac{(k-1)c_{0,k}+1}{k+2}}$ \\
\vdots & \vdots & \vdots & \vdots \\
$n$ & $c_{0,n}$ & $2c_{0,n}$ & $c_{0,n}$\\
\end{tabular}$
\end{center}

The next top degree sequence is ${\bf \bar{d}^2}=(1,k+2,k+3)$ and its coefficient is 
\[ \ga_2 = \min\Big\{ \frac{k-3(c_{0,k}-1)}{k+2}, \frac{2c_{0,k+1}}{k+2}, \frac{c_{0,k+1}}{k+1}-\frac{(k-1)c_{0,k}+1}{(k+1)(k+2)}\Big\}.\]

Therefore, it splits into two sub-cases :\\


\underline{Case II.1}: $\ga_2 = \displaystyle\frac{k-3(c_{0,k}-1}{k+2}$ as
$k+2 < c_{0,k+1}+2c_{0,k}.$ Then

\begin{center}$\be(L,x)- \sum\limits_{i=0}^{2}\ga_{i}\pi_{{\bf \bar{d}^i}} = \small \begin{tabular}{c|ccc}
$ $ & 0 & 1 & 2\\
\hline
$k$ & $c_{0,k} $ & - & -  \\
$k+1$ & $c_{0,k+1}$ & $2c_{0,k+1}-k+3(c_{0,k}-1)$ &  $c_{0,k+1}+2c_{0,k}-(k+2)$ \\
\vdots & \vdots & \vdots & \vdots \\
$n$ & $c_{0,n}$ & $2c_{0,n}$ & $c_{0,n}$\\
\end{tabular}$
\end{center}

The next degree sequence is ${\bf \bar{d}^3}=(k,k+2,k+3)$ and its coefficient is
 \[\ga_3 = \min \Big\{ c_{0,k}, \frac{2}{3}c_{0,k+1}-\frac{k}{3}+c_{0,k}-1, \frac{1}{2}c_{o,k+1} + c_{0,k}- \frac{k+2}{2} \Big\}.\]
If $c_{0,k} <  \frac{1}{2}c_{o,k+1} + c_{0,k}- \frac{k+2}{2}$, then it implies $k+2 < c_{0,k+1}$ which is not possible.
Thus the coefficient must be  \[\ga_3 = \displaystyle{\frac{c_{0,k+1}}{k+1}-\frac{(k-1)c_{0,k}}+1{(k+1)(k+2)}}.\] Then the remaining diagram looks like

\begin{center}$\be(L,x)- \sum\limits_{i=0}^{3}\ga_{i}\pi_{{\bf \bar{d}^i}} = \small \begin{tabular}{c|ccc}
$ $ & 0 & 1 & 2\\
\hline
$k$ & $\frac{k+2-c_{0,k+1}}{2}$ & - &  - \\
$k+1$ & $c_{0,k+1}$ & $\frac{c_{0,k+1}+k}{2}$ &  - \\
\vdots & \vdots & \vdots & \vdots \\
$n$ & $c_{0,n}$ & $2c_{0,n}$ & $c_{0,n}$\\
\end{tabular}$
\end{center}


\underline{Case II.2}: Note that if $G_{max}(L,x) = k+1$ and 
$\ga_2 = \displaystyle{\frac{c_{0,k+1}}{k+1}-\frac{(k-1)c_{0,k}+1}{(k+1)(k+2)}}$, then the algorithm contradicts the assumption $c_{0,k}\geq 1$.  So this case does not exist if the maximum degree is $k+1$. Suppose $G_{max}(L,x) > k+1$, so \[\ga_2 = \displaystyle{\frac{c_{0,k+1}}{k+1}-\frac{(k-1)c_{0,k}+1}{(k+1)(k+2)}} \, \text{ if } \,  2c_{0,k} + c_{0,k+1} < k+2.\]

Then the remaining diagram is

\begin{center}$\be(L,x)- \sum\limits_{i=0}^{2}\ga_{i}\pi_{{\bf \bar{d}^i}} = \small \begin{tabular}{c|ccc}
$ $ & 0 & 1 & 2\\
\hline
$1$ & $1-\frac{(2c_{0,k}+c_{0,k+1})-1}{k+1}$ & - &  - \\
\vdots & \vdots & \vdots & \vdots \\
$k$ & $c_{0,k} $ & -  & -  \\
$k+1$ & $c_{0,k+1}$ & $\frac{kc_{0,k+1}+(k-1)c_{0,k}+1}{k+1}$ &  - \\
$k+2$ & $c_{0,k+2}$ & $2c_{0,k+2}$ & $c_{0,k+2}$\\
\vdots & \vdots & \vdots & \vdots \\
$n$ & $c_{0,n}$ & $2c_{0,n}$ & $c_{0,n}$\\
\end{tabular}$
\end{center}


The top degree sequence of the remaining diagrams is $(1,k+2,k+4)$ and its coefficient is 
$$\ga_3 = \min\Big\{ \displaystyle\frac{1}{2}-\displaystyle\frac{2c_{0,k}+c_{0,k+1}-1}{2(k+1)}, \displaystyle\frac{kc_{0,k+1}+(k-1)c_{0,k}+1}{(k+1)(k+3)}, \displaystyle\frac{c_{0,k+2}}{k+1}\Big\}.$$ Next we observe each possible sub-cases for $\ga_3$.

\underline{Case II.2.a }: If $\ga_3 = \displaystyle\frac{1}{2}-\displaystyle\frac{2c_{0,k}+c_{0,k+1}-1}{2(k+1)}$ as a result of 
\[ k+4 < 3c_{0,k+1}+4c_{0,k} \text{ and } k+2 < 2c_{0,k}+c_{0,k+1}+2c_{0,k+2}. \]
Then we get the remaining diagram 
\begin{center}$\be(L,x)- \sum\limits_{i=0}^{3}\ga_{i}\pi_{{\bf \bar{d}^i}} = \small \begin{tabular}{c|ccc}
$ $ & 0 & 1 & 2\\
\hline
$k$ & $c_{0,k} $ & -  & -  \\
$k+1$ & $c_{0,k+1}$ & $\frac{3c_{0,k+1}+4c_{0,k}-(k+4)}{2}$ &  - \\
$k+2$ & $c_{0,k+2}$ & $2c_{0,k+2}$ & $c_{0,k+2}-\frac{c_{0,k+1+2c_{0,k}}-1}{2}-\frac{k+1}{2}$\\
\vdots & \vdots & \vdots & \vdots \\
$n$ & $c_{0,n}$ & $2c_{0,n}$ & $c_{0,n}$\\
\end{tabular}$
\end{center}

In this case, we would like to continue to decompose one more step. The coefficient of the pure diagram $\pi_{(k,k+2,k+4)}$ comes from 
\[\ga_4=\min\Big\{ c_{0,k}, \displaystyle\frac{3c_{0,k+1}+4c_{0,k}-(k+4)}{4}, c_{0,k+2}+\displaystyle\frac{c_{0,k+1}+2c_{0,k}-2-k}{2}\Big\}.\]

 (a1):\, If $\ga_4 = c_{o,k}$, that is,
\[\displaystyle\frac{k+4}{3} < c_{0,k+1} \text{ and } k+2 < c_{0,k+1}+2c_{0,k+2}\]

So the remaining diagram becomes

\begin{center}$\be(L,x)- \sum\limits_{i=0}^{4}\ga_{i}\pi_{{\bf \bar{d}^i}} = \small \begin{tabular}{c|ccc}
$ $ & 0 & 1 & 2\\
\hline
$k+1$ & $c_{0,k+1}$ & $\frac{3c_{0,k+1}-(k+4)}{2}$ &  - \\
$k+2$ & $c_{0,k+2}$ & $2c_{0,k+2}$ & $c_{0,k+2}-\frac{k+2-c_{0,k+1}}{2}$\\
\vdots & \vdots & \vdots & \vdots \\
$n$ & $c_{0,n}$ & $2c_{0,n}$ & $c_{0,n}$\\
\end{tabular}$
\end{center}

(a2):\, If $\ga_4 = \displaystyle\frac{3c_{0,k+1}+4c_{0,k}-(k+4)}{4}$, which is forced by
\[\displaystyle\frac{k+4}{3} > c_{0,k+1} \text{ and }  c_{0,k+1} +k < 4c_{0,k+2.}\]

Then 
\begin{center}$\be(L,x)- \sum\limits_{i=0}^{4}\ga_{i}\pi_{{\bf \bar{d}^i}} = \small \begin{tabular}{c|ccc}
$ $ & 0 & 1 & 2\\
\hline
$k$ & $\frac{k+4-3c_{0,k+1}}{4} $ & -  & -  \\
$k+1$ & $c_{0,k+1}$ & - &  - \\
$k+2$ & $c_{0,k+2}$ & $2c_{0,k+2}$ & $c_{0,k+2}-\frac{k+c_{0,k+1}}{4}$\\
\vdots & \vdots & \vdots & \vdots \\
$n$ & $c_{0,n}$ & $2c_{0,n}$ & $c_{0,n}$\\
\end{tabular}$
\end{center}

(a3):\, If $\ga_4 = c_{0,k+2}+\displaystyle\frac{c_{0,k+1}+2c_{0,k}-2-k}{2}$, which is caused by
\[ 2c_{0,k+2}+c_{0,k+1} < k+2 \text{ \, and \, } 4c_{0,k+2} < k + c_{0,k+1}.\]
Then the remaining diagram becomes

\begin{center}$\be(L,x)- \sum\limits_{i=0}^{4}\ga_{i}\pi_{{\bf \bar{d}^i}} = \small \begin{tabular}{c|ccc}
$ $ & 0 & 1 & 2\\
\hline
$k$ & $\frac{k+2-2c_{0,k+2}-c_{0,k+1}}{2}$ & -  & -  \\
$k+1$ & $c_{0,k+1}$ & $\frac{c_{0,k+1}-4c_{0,k+2}+k}{2}$ &  - \\
$k+2$ & $c_{0,k+2}$ & $2c_{0,k+2}$ & -\\
$k+3$ & $c_{0,k+3}$ & $2c_{0,k+3}$ & $c_{0,k+3}$\\
\vdots & \vdots & \vdots & \vdots \\
$n$ & $c_{0,n}$ & $2c_{0,n}$ & $c_{0,n}$\\
\end{tabular}$
\end{center}

\underline{Case II.2.b }: If $\ga_3 = \displaystyle\frac{kc_{0,k+1}+(k-1)c_{0,k}+1}{(k+1)(k+3)}$ because of the following inequalities
\[3c_{0,k+1}+4c_{0,k} < k+4 \text{ and }  kc_{0,k+1}+(k-1)c_{0,k}+1 < (k+3)c_{0,k+2}.\]
Thus, the remaining diagram is
$$\be(L,x)- \sum\limits_{i=0}^{3}\ga_{i}\pi_{{\bf \bar{d}^i}} $$
\begin{center}$= \small \begin{tabular}{c|ccc}
$ $ & 0 & 1 & 2\\
\hline
$1$ & $\frac{(k+1)(k+2)-2c_{0,k}(2k+1)+3c_{0,k+1}(k-1)}{(k+1)(k+2)}$ & - &  - \\
\vdots & \vdots & \vdots & \vdots \\
$k$ & $c_{0,k} $ & -  & -  \\
$k+1$ & $c_{0,k+1}$ & - &  - \\
$k+2$ & $c_{0,k+2}$ & $2c_{0,k+2}$ & $\frac{(k+3)c_{0,k+2}-kc_{0,k+1}-(k-1)c_{0,k}-1}{k+3}$\\
\vdots & \vdots & \vdots & \vdots \\
$n$ & $c_{0,n}$ & $2c_{0,n}$ & $c_{0,n}$\\
\end{tabular}$
\end{center}

\underline{Case II.2.c }: If $\ga_3 = \displaystyle\frac{1}{2}-\displaystyle\frac{2c_{0,k}+c_{0,k+1}-1}{2(k+1)}$ as
\[ 2c_{0,k+2}+c_{0,k+1}+2c_{0,k} < k+2 \text{ and }  (k+3)c_{0,k+2} < kc_{0,k+1}+(k-1)c_{0,k}+1. \]
Then $$\be(L,x)- \sum\limits_{i=0}^{2}\ga_{i}\pi_{{\bf \bar{d}^i}}$$
\begin{center}$ = \tiny \begin{tabular}{c|ccc}
$ $ & 0 & 1 & 2\\
\hline
$1$ & $\boxed{1- \frac{2c_{0,k}+c_{0,k+1}-1}{k+1} }- \frac{2c_{0,k+2}}{k+1}$ & - &  - \\
\vdots & \vdots & \vdots & \vdots \\
$k$ & $c_{0,k} $ & -  & -  \\
$k+1$ & $c_{0,k+1}$ & $\boxed{\frac{kc_{0,k+1}+(k-1)c_{0,k}+1}{k+1}} - \frac{(k+3)c_{0,k+2}}{k+1}$ &  - \\
$k+2$ & $c_{0,k+2}$ & $2c_{0,k+2}$ & $\boxed{c_{0,k+2}}-\frac{(k+1)c_{0,k+2}}{k+1} = 0$\\
$k+3$ & $c_{0,k+3}$ & $2c_{0,k+3}$ & $c_{0,k+3}$\\
\vdots & \vdots & \vdots & \vdots \\
$n$ & $c_{0,n}$ & $2c_{0,n}$ & $c_{0,n}$\\
\end{tabular}$
\end{center}
We notice that the pattern of the above remaining diagram is similar to the one in the beginning of the Case II.2. All possible top degree sequences $(*, *, k+4)$ in Case II.2.a and Case II.2.b will be replaced by $(*, *, k+5)$ in the next steps of Case II.2.c.\\
{\bf Case III:} If $G_{max}(L,x)= k+1$ then by Eliahou-Kervaire formula we get $c_{0,k}+c_{0,k+1} = k+1$ as $L$ is an Artinian lex segment ideal. Let $\ga_1 = \frac{c_{0,k+1}}{k}$, then it requires that $2c_{0,k+1}+c_{0,k} < k+1$ and so $c_{0,k+1} <1$ which is a contradiction. Thus this case does not exits if the maximum degree is $k+1$. Suppose that $G_{max}(L,x) >k+1$.  Thus  
\[\ga_1 = \displaystyle\frac{c_{0,k+1}}{k}\] since
\[2c_{0,k+1}+c_{0,k}<k+1 \text{ and } (k+2)c_{0,k+1}<(k-1)c_{0,k}+1.\]

Then the remaining diagram after subtracting two pure diagrams with corresponding coefficients is
$$\be(L,x)- \left(\displaystyle \frac{c_{0,k}-1}{k} \right)\pi_{(1,k+1,k+2)}-\left(\displaystyle\frac{c_{0,k+1}}{k}\right)\pi_{(1,k+1,k+3)} \, \, =$$

\begin{center}$ \small \begin{tabular}{c|ccc}
$ $ & 0 & 1 & 2\\
\hline
$1$ & $\boxed{\frac{k+1-(c_{0,k}+2c_{o,k+1})}{k}}$ & - &  - \\
\vdots & \vdots & \vdots & \vdots \\
$k$ & $c_{0,k} $ &  $\boxed{\frac{(k-1)c_{0,k}+1-(k+2)c_{0,k+1}}{k}}$ & -  \\
$k+1$ & $c_{0,k+1}$ & $2c_{0,k+1}$ &  - \\
$k+2$ & $c_{0,k+2}$ & $2c_{0,k+2}$ &  $\boxed{c_{0,k+2}}$ \\
\vdots & \vdots & \vdots & \vdots \\
$n$ & $c_{0,n}$ & $2c_{0,n}$ & $c_{0,n}$\\
\end{tabular}$
\end{center}

Then the next top degree sequence is ${\bf \bar{d}^2}=(1,k+1,k+4)$ and the coefficient 
\[\ga_2 = \min\left\{\displaystyle\frac{k+1-(c_{0,k}+2c_{o,k+1})}{3k},  \displaystyle\frac{(k-1)c_{0,k}+1-(k+2)c_{0,k+1}}{k(k+3)},  \displaystyle\frac{c_{0,k+2}}{k}\right\}.\]

\underline{Case III.1:} Let $\ga_2 = \displaystyle\frac{k+1-(c_{0,k}+2c_{o,k+1})}{3k}$. Then

\begin{center}$\be(L,x)- \sum\limits_{i=0}^{2}\ga_{i}\pi_{{\bf \bar{d}^i}} = \small \begin{tabular}{c|ccc}
$ $ & 0 & 1 & 2\\
\hline
$k$ & $c_{0,k} $ &  $\frac{(4(c_{0,k}-1)-(c_{0,k+1}+k)}{3}$ & -  \\
$k+1$ & $c_{0,k+1}$ & $2c_{0,k+1}$ &  - \\
$k+2$ & $c_{0,k+2}$ & $2c_{0,k+2}$ &  $c_{0,k+2}-\frac{(k+1)-(2c_{0,k+1}+c_{0,k})}{3}$ \\
\vdots & \vdots & \vdots & \vdots \\
$n$ & $c_{0,n}$ & $2c_{0,n}$ & $c_{0,n}$\\
\end{tabular}$
\end{center}

Thus ${\bf \bar{d}^3}=(k,k+1,k+4)$ and 
\[
\ga_3 = \min\left\{ \displaystyle\frac{c_{0,k}}{3}, \displaystyle\frac{(4(c_{0,k}-1)-(c_{0,k+1}+k)}{12}, c_{0,k+2}-\displaystyle\frac{(k+1)-(2c_{0,k+1}+c_{0,k})}{3}\right\}.
\]
We next observe one step more in the decomposition. So we get the following sub-cases since$\ga_3$ has two possible cases:


\underline{Case III.1.a:} When $\ga_3 = \displaystyle\frac{c_{0,k}}{3}$, the remaining diagram turns into

\begin{center}$\be(L,x)- \sum\limits_{i=0}^{3}\ga_{i}\pi_{{\bf \bar{d}^i}} = \small \begin{tabular}{c|ccc}
$ $ & 0 & 1 & 2\\
\hline
$k$ & $1+ \frac{c_{0,k}+k}{4} $ &  - & -  \\
$k+1$ & $c_{0,k+1}$ & $2c_{0,k+1}$ &  - \\
$k+2$ & $c_{0,k+2}$ & $2c_{0,k+2}$ &  $c_{0,k+2}+\frac{3c_{0,k+1}-k}{4}$ \\
\vdots & \vdots & \vdots & \vdots \\
$n$ & $c_{0,n}$ & $2c_{0,n}$ & $c_{0,n}$\\
\end{tabular}$
\end{center}


\underline{Case III.1.b:} When $\ga_3 = c_{0,k+2}-\displaystyle\frac{(k+1)-(2c_{0,k+1}+c_{0,k})}{3}$, the remaining diagram has the form

\begin{center}$\be(L,x)- \sum\limits_{i=0}^{3}\ga_{i}\pi_{{\bf \bar{d}^i}} = \small \begin{tabular}{c|ccc}
$ $ & 0 & 1 & 2\\
\hline
$k$ & $k+1 -(3c_{0,k+2}+2c_{0,k+1})$ &  $k-(4c_{0,k+2}+3c_{0,k+1})$ & -  \\
$k+1$ & $c_{0,k+1}$ & $2c_{0,k+1}$ &  - \\
$k+2$ & $c_{0,k+2}$ & $2c_{0,k+2}$ &  - \\
$k+3$ & $c_{0,k+3}$ & $2c_{0,k+3}$ &  $c_{0,k+3}$\\
\vdots & \vdots & \vdots & \vdots \\
$n$ & $c_{0,n}$ & $2c_{0,n}$ & $c_{0,n}$\\
\end{tabular}$
\end{center}

 
\underline{Case III.2:} Let $\ga_2 =\displaystyle\frac{(k-1)c_{0,k}+1-(k+2)c_{0,k+1}}{k(k+3)}$ and then 
 \begin{center}$\be(L,x)- \sum\limits_{i=0}^{2}\ga_{i}\pi_{{\bf \bar{d}^i}} = \small \begin{tabular}{c|ccc}
$ $ & 0 & 1 & 2\\
\hline
$1$ & $\frac{k+4-(c_{0,k}+c_{o,k+1})}{k+3}$ & - &  - \\
\vdots & \vdots & \vdots & \vdots \\
$k$ & $c_{0,k} $ &  - & -  \\
$k+1$ & $c_{0,k+1}$ & $2c_{0,k+1}$ &  - \\
$k+2$ & $c_{0,k+2}$ & $2c_{0,k+2}$ &  $c_{0,k+2}-\frac{(k-1)c_{0,k}+1-(k+2)c_{0,k+1}}{k+3}$ \\
\vdots & \vdots & \vdots & \vdots \\
$n$ & $c_{0,n}$ & $2c_{0,n}$ & $c_{0,n}$\\
\end{tabular}$
\end{center}
 
\underline{Case III.3:}  Let $\ga_2 = \displaystyle\frac{c_{0,k+2}}{k}$ and then 
 \begin{center}$\be(L,x)- \sum\limits_{i=0}^{2}\ga_{i}\pi_{{\bf \bar{d}^i}} = \small \begin{tabular}{c|ccc}
$ $ & 0 & 1 & 2\\
\hline
$1$ & $\frac{k+1-(c_{0,k}+2c_{o,k+1}+3c_{0,k+2})}{k}$ & - &  - \\
\vdots & \vdots & \vdots & \vdots \\
$k$ & $c_{0,k} $ &  $\frac{(k-1)c_{0,k}+1-(k+2)c_{0,k+1}-(k+3)c_{0,k+2}}{k}$ & -  \\
$k+1$ & $c_{0,k+1}$ & $2c_{0,k+1}$ &  - \\
$k+2$ & $c_{0,k+2}$ & $2c_{0,k+2}$ &  -\\
$k+3$ & $c_{0,k+3}$ & $2c_{0,k+3}$ &  $c_{0,k+3}$\\
\vdots & \vdots & \vdots & \vdots \\
$n$ & $c_{0,n}$ & $2c_{0,n}$ & $c_{0,n}$\\
\end{tabular}$
\end{center}
Hence we could pause the decomposing process since we have observed enough part of BS decomposition of the Betti diagram $(L,x)$ so that we can compare each possible remaining diagram with the ones will be obtained from BS decompositon of the Betti diagram of $L$.

We examine the BS decomposition of the lex ideal $L$.
First of all, as a trivial case, we notice that if $G_{min}(L)=1$, then the statement is vacuously true since $L =(L,x)$. 

We next induct on the difference of the initial degrees $G_{min}(J)-G_{min}(\mfa)\geq 1$.\\
{\bf\textit{Base Step:}} In this step, we show that the statement is true for the lex ideals $L=x\mfa + J$ when $G_{min}(J)-G_{min}(\mfa)=1$. 
That is, if $G_{min}(J)=k\geq2$ then $G_{min}(\mfa)=k-1$. So $\mfa = (x,y,z)^{k-1}$ since $L$ is a lex ideal.
%
%
To this end, we modify the Betti diagram of $L$ in the table \ref{Betti L} to this particular lex ideal $L$;
\begin{center}
$ \be(L) = \begin{tabular}{c|ccc}
  $ $ & 0 & 1 & 2\\
  \hline
   $k$ & $\frac{k(k+1)}{2}+c_{0,k}$ & $(k-1)(k+1)+2c_{0,k}-1$ & $\frac{k(k-1)}{2}+c_{0,k}-1$\\
  $k+1$ & $c_{0,k+1} $ & $2c_{0,k+1}$ & $c_{0,k+1}$\\
\vdots & \vdots & \vdots & \vdots \\
$n$ & $c_{0,n}$ & $2c_{0,n}$ & $c_{0,n}$
  \end{tabular}$
 \end{center}
Obviously the first top degree sequence and its coefficient are ${\bf d^0}=(k,k+1,k+2)$ and $\al_0 = \displaystyle\frac{k(k-1)}{2}+c_{0,k}-1$, respectively.
 Then ${\bf d^1}=(k,k+1,k+3)$ becomes the next degree sequence with the coefficient 
 \[\al_1=\min\Big\{\displaystyle\frac{k}{3}, c_{0,k+1}\Big\}.\]
Now, analyze each possible cases for the next step in the decomposition. \\
(i)\, If $\al_1=\displaystyle\frac{k}{3} < c_{0,k+1}$ then the remaining diagram of of $\be(L)$ after three steps becomes
  \begin{center}$ \be(L) - \sum\limits_{i=0}^{1}\al_i\pi_{{\bf d^i}} = \begin{tabular}{c|ccc}
  $ $ & 0 & 1 & 2\\
  \hline
  $k$ & $\frac{k}{3}+1$ & - & -\\
  $k+1$ & $c_{0,k+1} $ & $2c_{0,k+1}$ & $c_{0,k+1}-\frac{k}{3}$ \\
  \vdots & \vdots & \vdots & \vdots \\
 $n$ & $c_{0,n}$ & $2c_{0,n}$ & $c_{0,n}$\\
  \end{tabular} $
  \end{center}  
 If $ c_{0,k} > \displaystyle\frac{k}{3}+1 \text{ and } 2c_{0,k+1}+c_{0,k} >k+1 $ as in the Case I of $(L,x)$ and since $\al_1=\displaystyle\frac{k}{3} < c_{0,k+1}$ this remaining diagrams matches with the one in the Case I.1,
  $$ \be(L) - \sum\limits_{i=0}^{1}\al_i\pi_{{\bf d^i}} = \be(L,x) - \sum\limits_{i=0}^{2}\ga_i\pi_{{\bf \bar{d}^i}}.$$  
 Hence we are done as the BS decompositions for $L$ and $(L,x)$ have the same  remaining diagram.

Otherwise the BS decomposition of $(L,x)$ results in either as in {\bf Case II} or  {\bf Case III}, we keep decomposing the Betti diagram of $L$. 
 
 The next degree sequence is ${\bf d^2}=(k,k+2,k+3)$ and the coefficient of $\pi_{(k,k+2,k+3)}$ is 
  \[\al_2 = \min\left\{\displaystyle\frac{k}{3}+1, \displaystyle\frac{2c_{0,k+1}}{3}, \displaystyle\frac{c_{0,k+1}}{2}-\frac{k}{6}\right\}.\]
As $ \frac{k}{3} +1 < \frac{c_{0,k+1}}{2}-\frac{k}{6}$ implies $k+2 < c_{0,k+1}$ which is not possible. So   \[\al_2 = \frac{c_{0,k+1}}{2}-\frac{k}{6}. \]
 Therefore the remaining diagram is
  
  \begin{center}$ \be(L) - \sum\limits_{i=0}^{2}\al_i\pi_{{\bf d^i}} = \begin{tabular}{c|ccc}
  $ $ & 0 & 1 & 2\\
  \hline
    $k$ & $\frac{k+2-c_{0,k+1}}{2}$ & - & -\\
  $k+1$ & $c_{0,k+1} $ & $\frac{c_{0,k+1}+k}{2}$ & -\\
   $k+2$ & $c_{0,k+2}$ & $2c_{0,k+2}$ & $c_{0,k+2}$\\
  \vdots & \vdots & \vdots & \vdots \\
 $n$ & $c_{0,n}$ & $2c_{0,n}$ & $c_{0,n}$\\
  \end{tabular}$ 
  \end{center}  
  
  If we have Case II for $(L,x)$, which means 
  \[\frac{k}{3} +1 > c_{0,k} \text{ and } (k-1)c_{0,k} +1 < (k+2)c_{0,k+1}\] 
  and also  $ k+2 < c_{0,k+1}+2c_{0,k},$ so
  
 \[\be(L) - \sum\limits_{i=0}^{2}\al_i\pi_{{\bf d^i}} = \be(L,x) - \sum\limits_{i=0}^{3}\gamma_i\pi_{e^i}, \text{ as in Case II.1.} \]
 
 If $k+2 > c_{0,k+1}+2c_{0,k}$, which corresponds to Case II.2 for $(L,x)$, again we move on to the next step. So ${\bf d^3}=(k,k+2,k+4)$ and then 
 \[\al_3 = \min\Big\{ \displaystyle\frac{k+2-c_{0,k+1}}{2}, \displaystyle\frac{c_{0,k+1}+k}{4}, c_{0,k+2}\Big\}.\] This splits into three cases;
 \begin{itemize}
 \item 
 $\be(L) - \sum\limits_{i=0}^{2}\al_i\pi_{{\bf d^i}}-  \frac{k+2-c_{0,k+1}}{2}\pi_{d^3} = 
 \be(L,x) - \sum\limits_{i=0}^{4}\gamma_i\pi_{e^i}, \text{ as in Case II.2.a1.}$.
 
 \item  $\be(L) - \sum\limits_{i=0}^{2}\al_i\pi_{{\bf d^i}}-  \frac{c_{0,k+1}+k}{4}\pi_{d^3} = 
 \be(L,x) - \sum\limits_{i=0}^{4}\gamma_i\pi_{e^i}, \text{ as in Case II.2.a2.}$.
 
  \item  $\be(L) - \sum\limits_{i=0}^{2}\al_i\pi_{{\bf d^i}}-  c_{0,k+2}\pi_{d^3} = 
 \be(L,x) - \sum\limits_{i=0}^{4}\gamma_i\pi_{e^i}, \text{ as in Case II.2.a3.}$.
 
 \end{itemize}
 
We might need to recall that Case II.2.a for $(L,x)$ requires $k+4 < 3c_{0,k+1}+4c_{0,k}$  and $k+2 < 2c_{0,k}+c_{0,k+1} + 2c_{0,k+2}$. If $k+4 > 3c_{0,k+1}+4c_{0,k}$ as in Case II.2.b, it contradicts the first assumption $\frac{k}{3} < c_{0,k+1}$.  As Case II.2.c is a case similar to Case II.2, continuing the algorithm will leads us again the same remaining diagram as we have Case II.2.a.

Thus, if either {\bf Case I} or {\bf Case II} holds for the decomposition of $(L,x)$, we always end up with the same remaining diagrams, even with the ones having size $3$.

If the decomposition for $(L,x)$ turns into the {\bf Case III}, we get 
\[ 2c_{0,k+1} + c_{0,k} < k+1 \text{\, and \, } (k+2)c_{0,k+1} < (k-1)c_{0,k} +1.\]
On the other hand, by the assumption $\frac{k}{3} < c_{0,k+1}$, the former and the latter inequalities imply $c_{0,k} < \frac{k}{3}+1$ and $\frac{k}{3} + 1 < c_{0,k}$, respectively and we have a contradiction. Thus, this situation cannot come true at all.\\
(ii)\, If $\al_1 = c_{0,k+1} < \frac{k}{3}$, and if the BS decomposition of $(L,x)$ follows {\bf Case I} then we have
  \begin{center}$\be(L) - \sum\limits_{i=0}^{1}\al_i\pi_{{\bf d^i}} = \begin{tabular}{c|ccc}
  $ $ & 0 & 1 & 2\\
  \hline
  $k$ & $k+1-2c_{0,k+1}$ & $k-3c_{0,k+1}$ & - \\
  $k+1$ & $c_{0,k+1} $ & $2c_{0,k+1}$ & -\\
  \vdots & \vdots & \vdots & \vdots \\
 $n$ & $c_{0,n}$ & $2c_{0,n}$ & $c_{0,n}$
  \end{tabular} = \be(L,x) - \sum\limits_{i=0}^{2}\ga_i\pi_{{\bf \bar{d}^i}}.$ 
  \end{center}
  
If the decomposition of $(L,x)$ follows different path, the next degree sequence in the decomposition for $L$ becomes ${\bf d} = (k,k+1,k+4)$ and the coefficient  is
\[\al_2 = \min\left\{ \displaystyle\frac{k+1-2c_{0,k+1}}{3},  \displaystyle\frac{k-3c_{0,k+1}}{4}, c_{0,k+2}\right\}.\]
We move on the next sub-cases for each possibility for the coefficient $\al_2$.
\begin{enumerate}
\item[(ii.a)]\, If $\al_2 = \displaystyle\frac{k-3c_{0,k+1}}{4}$ thanks to the inequality
        \[ k < 3c_{0,k+1} + 4c_{0,k+2},\] 
         we have
 
      \begin{center}$ \be(L) - \sum\limits_{i=0}^{2}\al_i\pi_{{\bf d^i}} = \begin{tabular}{c|ccc}
  $ $ & 0 & 1 & 2\\
  \hline
    $k$ & $\frac{k+4+c_{0,k+1}}{4}$ & - & -\\
  $k+1$ & $c_{0,k+1} $ & $2c_{0,k+1}$ & -\\
   $k+2$ & $c_{0,k+2}$ & $2c_{0,k+2}$ & $\frac{4c_{0,k+2}+3c_{0,k+1}-k}{4}$\\
  \vdots & \vdots & \vdots & \vdots \\
 $n$ & $c_{0,n}$ & $2c_{0,n}$ & $c_{0,n}$\\
  \end{tabular}$ 
  \end{center}  

 If the algorithm follows through as in the Case III for the decomposition of $(L,x)$, the relation $k < 3c_{0,k+1} + 4c_{0,k+2}$ yields the same the remaining diagram. Hence
  
  \[\be(L) - \sum\limits_{i=0}^{2}\al_i\pi_{{\bf d^i}}  =  \be(L,x)- \sum\limits_{i=0}^{3}\ga_{i}\pi_{{\bf \bar{d}^i}}.\]
  
 If the Betti diagram of $(L,x)$ is decomposed as in the {\bf Case II}, we keep decomposing the Betti diagram of $L$ and get three sub-cases for the next coefficient
 
$$\al_3 = \min\Big\{ \displaystyle\frac{k+4+c_{0,k+1}}{4},  c_{0,k+1}, \displaystyle\frac{4c_{0,k+2}+3c_{0,k+1}-k}{4}\}.$$

 If $\al_3 =  \displaystyle\frac{k+4+c_{0,k+1}}{4}$, that is, 
 \[ k+4 < 3c_{0,k+1} \text{\, and\, } k+2 < 2c_{0,k+2}+c_{0,k+1}\]
 
 then
   \begin{center}$ \be(L) - \sum\limits_{i=0}^{3}\al_i\pi_{{\bf d^i}} = \begin{tabular}{c|ccc}
  $ $ & 0 & 1 & 2\\
  \hline
  $k+1$ & $c_{0,k+1} $ & $\frac{3c_{0,k+1}-k-4}{2}$ & -\\
   $k+2$ & $c_{0,k+2}$ & $2c_{0,k+2}$ & $c_{0,k+2-}\frac{k+2c_{0,k+1}}{2}$\\
  \vdots & \vdots & \vdots & \vdots \\
 $n$ & $c_{0,n}$ & $2c_{0,n}$ & $c_{0,n}$\\
  \end{tabular}$
  \end{center}  
 which is equal to  $\be(L,x)- \sum\limits_{i=0}^{4}\ga_{i}\pi_{{\bf \bar{d}^i}}$ since the same relations are also required for the Case II.2.a.1 for the decomposition of $(L,x)$.
  
 If $\al_3 = c_{0,k+1}$, that is, the relations are
 \[ 3c_{0,k+1} <k+4  \text{\, and\, } k < 4c_{0,k+2}-c_{0,k+1},\]
 
 then 
  
      \begin{center}$ \be(L) - \sum\limits_{i=0}^{3}\al_i\pi_{{\bf d^i}} = \begin{tabular}{c|ccc}
  $ $ & 0 & 1 & 2\\
  \hline
    $k$ & $\frac{k+4-3c_{0,k+1}}{4}$ & - & -\\
  $k+1$ & $c_{0,k+1} $ & - & -\\
   $k+2$ & $c_{0,k+2}$ & $2c_{0,k+2}$ & $\frac{4c_{0,k+2}-c_{0,k+1}-k}{4}$\\
  \vdots & \vdots & \vdots & \vdots \\
 $n$ & $c_{0,n}$ & $2c_{0,n}$ & $c_{0,n}$\\
  \end{tabular}$ 
  \end{center}  
and it is the same remaining diagram as in Case II.2.a.2 for $(L,x)$ because of the same required relations.
 
If $\al_3 =  \displaystyle\frac{4c_{0,k+2}+3c_{0,k+1}-k}{4}$, which is a consequence of the following inequalities
 \[ k+2 < 2c_{0,k+2}+c_{0,k+1}   \text{\, and\, } 4c_{0,k+2}-c_{0,k+1} < k.\] 
 Then 
  \begin{center}$ \be(L) - \sum\limits_{i=0}^{3}\al_i\pi_{{\bf d^i}} = \begin{tabular}{c|ccc}
  $ $ & 0 & 1 & 2\\
  \hline
    $k$ & $\frac{k+2-(2c_{0,k+2}+c_{0,k+1})}{2}$ & - & -\\
  $k+1$ & $c_{0,k+1} $ & $\frac{c_{0,k+1}-4c_{0,k+2}+k}{2}$& -\\
   $k+2$ & $c_{0,k+2}$ & $2c_{0,k+2}$ & -\\
    $k+3$ & $c_{0,k+3}$ & $2c_{0,k+3}$ & $c_{0,k+3}$\\
  \vdots & \vdots & \vdots & \vdots \\
 $n$ & $c_{0,n}$ & $2c_{0,n}$ & $c_{0,n}$\\
  \end{tabular}$ 
  \end{center}  
  $$= \be(L,x) - \sum\limits_{i=0}^{4}\ga_i\pi_{{\bf \bar{d}^i}}\text{ \, as in Case II.2.a.3 for $(L,x)$.}$$
 
\item[(ii.b)]\, If $\al_2 = c_{0,k+2}$ then the remaining diagram  becomes
 
      \begin{center}$ \be(L) - \sum\limits_{i=0}^{2}\al_i\pi_{{\bf d^i}} = \begin{tabular}{c|ccc}
  $ $ & 0 & 1 & 2\\
  \hline
    $k$ & $k+1-2c_{0,k+1}-3c_{0,k+2}$ & $k-3c_{0,k+1}-4c_{0,k+2}$ & -\\
  $k+1$ & $c_{0,k+1} $ & $2c_{0,k+1}$ & -\\
   $k+2$ & $c_{0,k+2}$ & $2c_{0,k+2}$ & - \\
    $k+3$ & $c_{0,k+3}$ & $2c_{0,k+3}$ & $c_{0,k+3}$ \\
  \vdots & \vdots & \vdots & \vdots \\
 $n$ & $c_{0,n}$ & $2c_{0,n}$ & $c_{0,n}$\\
  \end{tabular}$ 
  \end{center}  
  The above diagram matches with remaining diagrams as in the Case III.1.b for the decomposition of $(L,x)$ if $(L,x)$ decomposes as in {\bf Case III}. Now suppose that BS decomposition of $\be(L,x)$ after several steps ends up as in the {\bf Case II}. We want to show that same remaining diagram occurs for $L$ as well. If $G_{max}$ is not $k+3$, we may have another case which gives us similar pattern like above diagram such as
     \begin{center}$ \be(L) - \sum\limits_{i=0}^{2}\al_i\pi_{{\bf d^i}} -c_{0,k+3}\pi_{(k,k+1,k+5)}= \begin{tabular}{c|ccc}
  $ $ & 0 & 1 & 2\\
  \hline
    $k$ & $k+1-2c_{0,k+1}-3c_{0,k+2}-4c_{0,k+3}$ & $k-3c_{0,k+1}-4c_{0,k+2}-5c_{0,k+3}$ & -\\
  $k+1$ & $c_{0,k+1} $ & $2c_{0,k+1}$ & -\\
   $k+2$ & $c_{0,k+2}$ & $2c_{0,k+2}$ & - \\
    $k+3$ & $c_{0,k+3}$ & $2c_{0,k+3}$ & - \\
        $k+3$ & $c_{0,k+4}$ & $2c_{0,k+4}$ & $c_{0,k+4}$ \\
  \vdots & \vdots & \vdots & \vdots \\
 $n$ & $c_{0,n}$ & $2c_{0,n}$ & $c_{0,n}$\\
  \end{tabular}$ 
  \end{center}  
  
  Thus, assume the maximum degree is $k+3$. Therefore the next coefficient is 
  $\al_3 = \displaystyle\frac{k-(3c_{0,k+1}+4c_{0,k+2})}{5}$ with the pure diagram $\pi_{k,k+1,k+5}$.  Then the remaining diagram turns into

  \begin{center}$\be(L)- \sum\limits_{i=0}^{3}\al_{i}\pi_{{\bf d^i}} = \small \begin{tabular}{c|ccc}
$ $ & 0 & 1 & 2\\
\hline
$k$ & $\frac{k+5+2c_{0,k+2}+c_{0,k+1}}{2}$ & -  & -  \\
$k+1$ & $c_{0,k+1}$ & $2c_{0,k+1}$ &  - \\
$k+2$ & $c_{0,k+2}$ & $2c_{0,k+2}$ & -\\
$k+3$ & $c_{0,k+3}$ & $2c_{0,k+3}$ & $\frac{5c_{0,k+3}+4c_{0,k+2}+3c_{0,k+1}-k}{5}$
\end{tabular}$
\end{center}
  
  Also, we have assumed that the Betti diagram of $(L,x)$ decomposes as in {\bf Case II}.  We notice that the entries of the remaining diagram is closely related to the one in Case II.2.a3.
  
  \begin{center}$\be(L)- \sum\limits_{i=0}^{3}\al_{i}\pi_{{\bf d^i}} = \small \begin{tabular}{c|ccc}
$ $ & 0 & 1 & 2\\
\hline
$k$ & $\frac{k+2-2c_{0,k+2}-c_{0,k+1}}{2} + 3D$ & -  & -  \\
$k+1$ & $c_{0,k+1}$ & $\frac{c_{0,k+1}-4c_{0,k+2}+k}{2}+5D$ &  - \\
$k+2$ & $c_{0,k+2}$ & $2c_{0,k+2}$ & -\\
$k+3$ & $c_{0,k+3}$ & $2c_{0,k+3}$ & $c_{0,k+3}+2D$\\
\end{tabular} $
\end{center}
where $D = \frac{4c_{0,k+2}+3c_{0,k+1}-k}{10}$. Thus the next coefficient is
  \[ \al_4 = \min\left\{ \frac{k+2-2c_{0,k+2}-c_{0,k+1}}{6} + D, \frac{c_{0,k+1}-4c_{0,k+2}+k}{10}+D, \frac{c_{0,k+3}}{2}+D\right\}\]
  which follows the same paths for the decomposition of $(L,x)$ in Case II.2.a3.
  
\end{enumerate}

BS decompositions of $\be(L)$ and $\be(L,x)$ always come up with the same remaining diagrams after several steps of the decomposition. Moreover, we observe that they share not only the length two pure diagrams and their coefficients but also some length three pure diagrams.
 
Hence, in every possible case we end up with the same remaining diagrams for both $\be(L)$ and $\be(L,x)$. In other words, the BS decompositions of $L$ and $(L,x)$ coincide precisely after several steps of the algorithm. Thus the statement holds for the case of $G_{min}(J)-G_{min}(\mfa)=1$.\\
{\bf\textit{Induction Hypothesis:}} Let the statement be true for all lex ideals
$L = x\mfa +J$ with $G_{min}(J)-G_{min}(\mfa) = N \geq 1$.  
We need to show that it is also true for the lex ideals satisfying $G_{min}(J)-G_{min}(\mfa) = N+1$. To this end, we identify the initial degrees of $J$ and $\mfa$ by $G_{min}(J) = k$ and $G_{min}(\mfa)=m$. 

Suppose that $L = x\mfa + J$ is a lex ideal such that $k-m = N+1$. So $k-m = N+1 \geq 2.$ We prove this into two cases.

\underline{\textit{Case A:}} If $y^m\notin \mfa$. Since $\mfa$ is a lex ideal, we write $\mfa = x\mfb + I$. 
Then we notice that $G_{min}(I)\neq k$. Otherwise it contradicts to $y^k \in G(J)$. Thus $k \gneq G_{min}(I) \gneq m$ as $y^m \notin \mfa$. 
Define $\tilde{\mfa} \subset \mfa$ as the ideal containing all monomials of $\mfa$ of degree greater or equal to $m+1$ and note that $\tilde{\mfa}$ is also a lex ideal with $G_{min}(\tilde{\mfa}) = m+1$.
Define $\tilde{L} = x\tilde{\mfa} + J$ and it is a lex ideal with $G_{min}(J) - G_{min}(\tilde{\mfa}) = k - (m+1) = k-m-1 = N+1-1 = N$.
Therefore, by the induction hypothesis, $\be(\tilde{L})$ and $\be(L,x)$ have the same ends in their BS decompositions, i.e. same pure diagrams of length 
less than $3$ with same coefficients,
$$\be(\tilde{L}) - \sum\limits_{\tilde{{\bf d}}^i, \ \text{all length $3$ degree sequences}}\tilde{\al_i}\pi_{\tilde{{\bf d}}^i} =
\be(L,x) - \sum\limits_{{\bf d^i}, \ \text{all length $3$ degree sequences}}\ga_i\pi_{{\bf d}^i}.$$
On the other hand, $\tilde{\mfa}$ can be decomposed as $\tilde{\mfa} = x\tilde{\mfb} + \tilde{I}$. It is easy to see that $\tilde{I} = I$ as 
$y^m\notin \mfa$ and $G_{min}(\tilde{\mfb})=m$.
Clearly, $G_{min}(I)-G_{min}(\tilde{\mfb}) \leq (k-1)-m = N$. Again by the induction hypothesis BS decompositions of $\be(\tilde{\mfa})$ and  $\be(I,x)$ have the same ends.

Recall that $\mfa = x\mfb + I$, so we get
$G_{min}(I)-G_{min}(\mfb)\leq (k-1)-(m-1)= k-m = N+1$. 

Suppose that $G_{min}(I)-G_{min}(\mfb) < N+1$, then thanks to the induction hypothesis, BS decompositions of $\mfa$
and  $(I,x)$ have the same ends, so do $\mfa$ and $\tilde{\mfa}$. That is,
\[D:=\be(\mfa) - \sum \text{All length $3$ pure diagrams} = \be(\tilde{\mfa}) - \sum\text{All length $3$ pure diagrams}.\]
Also using the Theorem \ref{thm1} BS decompositions for the ideals $L$ and $\tilde{L}$ can be observed as;
$$\be(L) = \sum\limits_{{\bf d^i} \ \mbox{with\ } l({\bf d^i}) =3}\al_i\pi_{{\bf d^i}}  \ \ +  \ \ 
\begin{tabular}{c|ccc}
  & 0 & 1 & 2\\
\hline
$2$&    & \mbox{Remaining}  & \\  
\vdots &  & \mbox{diagram, $\mathbf{D}$} &  \\
\hdashline
$k$ &   &   &      \\
\vdots   &   & $\be_{i,i+j}(L,x),\ i\geq k$  &  \\
         &   &             & 
\end{tabular} $$
and
$$\be(\tilde{L}) = \sum\limits_{{\bf \tilde{d}}^i \ \mbox{with\ } l({\bf \tilde{d}}^i) =3}\tilde{\al}_i\pi_{{\bf \tilde{d}}^i}  \ \ +  \ \ 
\begin{tabular}{c|ccc}
  & 0 & 1 & 2\\
\hline
$2$&    & \mbox{Remaining}  & \\  
\vdots &  & \mbox{diagram, $\mathbf{D}$} &  \\
\hdashline
$k$ &   &   &      \\
\vdots   &   & $\be_{i,i+j}(L,x),\ i\geq k$  &  \\
         &   &             & 
\end{tabular} $$
This shows that $\be(L)$ and $\be(\tilde{L})$ have same ends and we also know that $\be(\tilde{L})$ and $\be(L,x)$ have the same ends. Hence the statement is true.

It remains to study when $G_{min}(I)-G_{min}(\mfb)=(k-1)-(m-1)= k-m = N+1$, which means $G_{min}(I)=k-1$. It follows from $G_{min}(J)=k$ that $I=(y,z)^{k-1}$. 
Then $\mfa= x\mfb+I$ and $\mfb = x\overline{\mfb}+\overline{I}$ where 
$G_{min}(\overline{I})-G_{min}(\overline{\mfb}) \geq (k-2)-(m-2) \geq N+1$. If it is a strict inequality then applying the same process as we have done for $L$ can be applied to $\mfa$ 
to prove the statement. If there is an equality, we end up with the same situation.
$L=x\mfa+J$ where $G_{min}(J)=k, \, G_{min}(\mfa)=m$ and $k-m =N+1$, and $\mfa=x\mfb+I$ where $I=(y,z)^{k-1}, G_{min}(\mfb)=m-1$, and $\mfb = x\overline{\mfb}+\overline{I}$ where $\overline{I}=(y,z)^{k-2}, G_{min}(\overline{\mfb})=m-2$.
We repeat this until we get
$$\mfc =x(x,y,z^{t-1})+ K \ \ \text{where}\ K=(y,z)^s, s=k-m+1, 1\leq t \leq k-m.$$  
For this form of the lex ideal, one can check the BS decomposition of the ideal $\mfc$.

$$\begin{matrix}\be(\mfc) & = & \begin{tabular}{c|ccc}
  $ $ & 0 & 1 & 2\\
  \hline
  $2$ & $2$ & $1$ & -\\
  \vdots & \vdots & \vdots & \vdots \\
  $t$ & $1$ & $2$ & $1$ \\
  \vdots & \vdots & \vdots & \vdots \\
 $s$ & $s+1$ & $2s+1$ & $s$\\
  \end{tabular}& & & & \\ \\
  & = &  \frac{1}{t}\begin{bmatrix}
  2: & t-1 & t & -\\
  t: & - & - & 1
  \end{bmatrix} & + &
  \frac{1}{t}\begin{bmatrix}
  2: & 1 & - & -\\
  t: & - & t-1 & t
  \end{bmatrix} \\ \\
  & + &
 \frac{1}{s}\begin{bmatrix}
  2: & s-t+1 & - & -\\
  t: & - & s & - \\
  s: & - & - & t-1
  \end{bmatrix}
   & + &  \frac{t-1}{s}\begin{bmatrix}
  2: & 1 & - & -\\
  s: & - & s & s-1
  \end{bmatrix}\\ \\
  & + & 
  1\begin{bmatrix}
  t: & 1 & - & -\\
  s: & - & s-t+2 & s-t+1
  \end{bmatrix} & + & 
   \underbrace{s\begin{bmatrix}
  s: & 1 & 1\\
  \end{bmatrix}+
  1\begin{bmatrix}
  s: & 1
  \end{bmatrix}}_{\text{same end as in the decomposition of} \ (L,x)}
  \end{matrix}
  $$ 
Therefore the statement is true for the ideal $\mfc$. So we may assume that, without loss of generality, $\mfa$ is in the form of $\mfc$, 
i.e., $G_{min}(I)-G_{min}(\mfa)< N+1$. This observation completes the proof for Case I.\\

\underline{\textit{Case B:}} Let $y^m \in \mfa$.
\begin{enumerate}
\item If $m \notin 1$; we write $\mfa = x\mfb+I$ and $G_{min}(I)=m>1$. This implies that $\mfb=(x,y,z)^{m-1}$. Consider $\tilde{\mfa}=(\mfa,x)=x(1)+I$. 
Clearly $\tilde{\mfa}=(\tilde{\mfa},x)$, so the statement is trivially true for the ideal $\tilde{\mfa}$.
Moreover, $G_{min}(I)-G_{min}(\mfa)=m-(m-1)=1$. By the base case, the decompositions of $\be(\mfa)$ and $\be(I,x)$ have the same ends. Hence,  
\begin{align*} \be(\mfa) - \sum \text{all length $3$ pure diagrams} & = &
\be(I,x) - \sum \text{all length $3$ pure diagrams}\\ 
& = & \be(\tilde{\mfa}) - \sum\text{all length $3$ pure diagrams}.
\end{align*}
Similar to the {\bf Case 1}, consider the lex ideal $\tilde{L}=x\tilde{\mfa}+J$ and $y^{G_{min}(\tilde{\mfa})}= y \notin \tilde{\mfa}$. Thus by the result of the {\bf Case 1}, 
the statement is true for $\tilde{L}$. We do exactly the same trick as in {\bf Case 1} to show 
 that $\be(L)$ and $\be(\tilde{L})$ have the same ends and it follows that the statement holds for $L$.
 \item If $m=1$; that is, $\mfa=(x,y,z^t)$ where $1\leq t \leq k-1$. 
 In the {\bf Case 1} we have already shown that the BS decomposition of the $\be(L)$ satisfy the statement if $L= x(x,y,z^t)+J$ where $J=(y,z)^k$.
 Nevertheless, for more general stable ideal $J\subset \mathbf{k}[y,z]$ we have already assumed that $L$ cannot be in that form in the statement. 
\end{enumerate}   

\begin{conjecture} 
The statement of Theorem \ref{thm2} holds for all Artinian lex-ideals in $\mathbf{k}[x,y,z]$. 
\end{conjecture}

Theorem \ref{thm2} shows that the ends of the Boij-S\"oderbeg decompositions of $L$ and $(L,x)=(J,x)$ are exactly the same for all Artinian lex ideals $L$ in $R$ except
the ones in the form of $ L = x(x,y,z^t) + J$ where $J$ is different from $(y,z)^{G_{min}(J)}$ and $1 < t < k-1$. 
On the other hand, based on the computations we have done using the \texttt{BoijSoederberg} packages of the computer algebra software \texttt{Macaulay2}, see \cite{M2}, we strongly believe that this result is also
true for the lex ideals in that particular form.

\section{Further Observations and Examples}
\label{observation}
For an Artinian lex ideal $L\subset \mathbf{k}[x,y,z]$ of codimension $3$, we have shown that the summands of length $3$ pure diagrams of the Boij-S\"{o}derberg decomposition of $\mfa$ where $\mfa = L : (x)$, and 
the summands of pure diagrams of length less than $3$ in the Boij-S\"{o}derberg decomposition of $(L,x)$ appear in the decomposition of 
the ideal $L=\mfa(x)+J$ in the beginning and in the end, respectively. 
\[
\be(L) = \left[ \begin{matrix}\text{length $3$ degree} \\ \text{sequences coming }\\ \text{from } \mfa(-1) \end{matrix} \right]  
+ \left[ \begin{matrix}\text{extra length $3$}\\ \text{degree sequences} \end{matrix}\right] + 
\left[ \begin{matrix} \text{all length $<3$  degree} \\ \text{sequences coming} \\ \text{from } (L,x) \end{matrix} \right].
\]
There might be also some other pure diagrams of length $3$ other than the ones coming from the BS decomposition of $\mfa$.
However, how this middle part containing pure diagrams of length $3$ comes out is not quite clear. One might ask whether or not the ideals $\mfb=L:(y)$ and $\mfc=L:(z)$ help to describe the middle part. In fact, examples show that there is a quite strong relation between them as well.
Nevertheless, both the results obtained in sections 3 and 4 about $\be(\mfa)$ and $\be(L,x)$ and the observations we discuss in this section about  $\be(\mfb)$ and $\be(\mfc)$ provide a very close approximation for the BS chain of degree sequences for $\be(L)$. Either our observation in this section is not enough to cover the entire middle part of the decomposition of $\be({L})$ or the BS decompositions of $\mfb$ and $\mfc$ may give redundant degree sequences.

Now in this section we illustrate the elation between the BS decompositions of the ideals $\mfb$, $\mfc$ and $L$ via examples.

\begin{example}  
Let \[
L = (x^2,xy^2,xyz,xz^2,y^8,y^7z,y^6z^2,y^5z^3,y^4z^4,y^3z^5,y^2z^6,yz^7,z^8)
\]
be a lex segment ideal in $R$. Then $\mfa =L:x = (x,y^2,yz,z^2)$ is lex segment ideal such that $L=x\mfa+J$ where $J=(y,z)^8$ is stable in $R$ and lex segment in $\mathbf{k}[y,z]$.
Similarly the ideals 
\[
\mfb =L:y = (x^2,xy,xz,y^7,y^6z,y^5z^2,y^4z^3,y^3z^4,y^2z^5,yz^6,z^7)=L:z=\mfc
\]
are lex segment ideals such that $L=y\mfb+I=z\mfc+K$ where $I=(x^2,xz^2,z^8)$ and $K=(x^2,xy^2,y^8)$.\\
We construct similar short exact sequences like in Lemma \ref{lemma for ses} for the ideals $\mfb$ and $\mfc$. Unlike the case for $\mfa$, we might have cancellations in the mapping cone of the short exact 
sequences for ideals. It means we can have cancellations in the Betti diagram since the mapping cone structure may not yield the minimal free resolution.
This situation causes different degree sequences that do not appear in BS decomposition of $L$. 

We first notice that $\mfb = \mfc$ and find the BS decomposition of $\be(\mfa)$

$$\be(\mfa)=(1)\pi_{(1,3,4)} + [\mbox{pure diags. of length}<3],$$

Then we consider the short exact sequence for the ideals  $\mfb$ and $L= y\mfb+I$

$$\minCDarrowwidth10pt\begin{CD}\label{mappingcone}
0 @>>>        I(-1) @>>>        \mfb(-1)\oplus I @>>> L @>>> 0 \\
@.            @AAA                 @AAA              @AAA     @.\\
@.   R(-3)\oplus R(-4)\oplus R(-9) @.  R(-2)\oplus R^4(-3)\oplus R^9(-8) @. R(-2)\oplus \overbrace{\cancel{R^4(-3)}}^{R^3(-3)}\oplus R^9(-8) @.\\
@.            @AAA                 @AAA             @AAA    @.\\
@.   R(-5)\oplus R(-10)     @.  R^4(-4)\oplus R^{16}(-9)      @.     \cancel{R(-3)}\oplus R^5(-4)\oplus R^{17}(-9) @.\\
@.            @AAA                 @AAA             @AAA     @.\\
@.             0     @.  R(-5)\oplus R^7(-10)      @.     R^2(-5)\oplus R^8(-10) @.\\
@.             @.                @AAA              @AAA     @.\\
@.             @.              0      @.        0     @.  @.
\end{CD}$$

The mapping cone of the short exact sequence for ideals $\mfb$ and $L $ (so the same for $\mfc$ and $L = z\mfc+K$) ends up with ``one'' cancellation in the first degree. So 
we interpret this as ignoring one pure diagram at the beginning, which is the one corresponding to the degree sequence $(2,3,4)$ at the beginning of the decomposition of $\be(\mfb)$.
Therefore,
$$\be(\mfb)=\be(\mfc)=\cancel{(1)\pi_{(2,3,4)}} + (\frac{1}{7})\pi_{(2,3,9)} + (\frac{8}{7})\pi_{(2,8,9)} + [\mbox{pure diags. length}<3].$$

The pure diagrams of length less than $3$ are coming from the ideal 
$$\be(L,x) = [\mbox{length $3$ pure diags.}] + (8)\pi_{(8,9)}+(1)\pi_{(8)}.$$
Hence we claim that the summands (with coefficients) in the BS decomposition of $\be(L)$ are
$$\be(L) \approx \underbrace{(1)\pi_{(2,4,5)}}_{\mbox{from}\ \mfa(-1)} + (\al_{2})\underbrace{\pi_{(3,4,10)}}_{\mbox{from}\ \mfb(-1)} + 
(\al_{3})\underbrace{\pi_{(3,9,10)}}_{\mbox{from}\ \mfb(-1)} + \underbrace{(8)\pi_{(8,9)} + (1)\pi_{(8)}}_{\mbox{from}\ (L,x)},$$
for some coefficients $\al_2, \al_3$ in $\mathbb{Q}$.
Indeed, the BS decomposition of $L$ is, 
$$\be(L) = (1)\pi_{(2,4,5)}+(\frac{2}{7})\pi_{(3,4,10)}+(\frac{9}{7})\pi_{(3,9,10)}+(8)\pi_{(8,9)}+(1)\pi_{(8)}.$$

The impressive point of this example is that  we are able to describe the entire BS chain of degree sequences of $L$ from its the colon ideals $\mfa$, $\mfb$
and the ideal $(L,x)$.  

\end{example}

\begin{example}\label{longer decomp}
This example shows that some different situations might occur other than the previous example. 

Let $L = (x^2, xy^2, xyz, xz^2, y^4, y^3z, y^2z^2, yz^6, z^9)$ be lex-segment ideal in $R$. Then 
$\mfa = L : x = (x,y^2, yz, z^2),  
\mfb = L : y = (x^2, xy, xz, y^3, y^2z, yz^2, z^6)$ and $\mfc = L : (z) = (x^2, xy, xz, y^3, y^2z, yz^5, z^8)$.
We observe that one cancellation occurs in the mapping cone process of each ideal $\mfb$ and $\mfc$. BS decompositions of $\mfa$, $\mfb$, $\mfc$  and $(L,x)$ are 
\begin{align*}
\be(\mfa) &= 1\pi_{(1,3,4)} + [\mbox{pure diags. of length }<3],\\
\be(\mfb) &= \cancel{1\pi_{(2,3,4)}} + \frac{1}{3}\pi_{(2,3,5)} + \frac{5}{6}\pi_{(2,4,5)} + \frac{1}{4}\pi_{(2,4,8)} + \frac{7}{20}\pi_{(3,4,8)} \\
&\, \, \, \, + \frac{1}{10}\pi_{(3,7,8)} +[\mbox{pure diags. of length }<3],\\
\be(\mfc) &= \cancel{1\pi_{(2,3,4)}} + \frac{1}{3}\pi_{(2,3,5)} + \frac{1}{3}\pi_{(2,4,5)} + \frac{1}{2}\pi_{(2,4,8)} + \frac{1}{10}\pi_{(3,4,8)} +  \frac{1}{10}\pi_{(3,7,8)} \\
&\, \, \, \, +\frac{3}{14}\pi_{(3,7,10)} + \frac{1}{42}\pi_{(3,9,10)} + [\mbox{pure diags. of length }<3]
\end{align*}
and
\[
\be(L,x) = [\mbox{pure diags. of length } 3] + 1\pi_{(4,10)} + 1\pi_{(7,10)} + 1\pi_{(9)}.
\]
So, the BS decomposition for the ideal $L$ is likely to be
\begin{align*}
\be(L)&\approx 1\pi_{(2,4,5)} + \al_2\pi_{(3,4,6)} + \al_3\pi_{(3,5,6)} + \al_4\pi_{(3,5,9)} + \al_5\pi_{(4,5,9)} \\
&\, \, \, \, \, + \al_6\pi_{(4,8,9)} + \al_7\pi_{(4,8,11)} + \mathbf{\al_8\pi_{(4,10,11)}}+ 1\pi_{(4,10)} + 1\pi_{(7,10)} + 1\pi_{(9)}
\end{align*}
where $\al_i \in \mathbb{Q} ,\ \ i=2,...8$. Thus it seems that we almost obtain the actual BS decomposition for $L$ which is
\begin{align*}
\be(L) & =\underbrace{1\pi_{(2,4,5)}}_{\mbox{from}\ \ \mfa(-1)} + \frac{2}{3}\underbrace{\pi_{(3,4,6)}}_{\mbox{from}\ \ \mfb(-1) \ \mbox{and}\ \mfc(-1)} + 
\frac{2}{3}\underbrace{\pi_{(3,5,6)}}_{\mbox{from}\ \ \mfb(-1) \ \mbox{and}\ \mfc(-1)}\\
 &\, \, \, \, \, + \frac{1}{2}\underbrace{\pi_{(3,5,9)}}_{\mbox{from}\ \ \mfb(-1) \ \mbox{and}\ \mfc(-1)}
+\frac{3}{10}\underbrace{\pi_{(4,5,9)}}_{\mbox{from}\ \ \mfb(-1) \ \mbox{and}\ \mfc(-1)} + \frac{1}{20}\underbrace{\pi_{(4,8,9)}}_{\mbox{from}\ \ \mfb(-1) \ \mbox{and}\ \mfc(-1)} \\
&\,\,\,\,\, +\frac{1}{4}\underbrace{\pi_{(4,8,11)}}_{\mbox{from}\ \ \mfc(-1)}
+ \underbrace{1\pi_{(4,10)} + 1\pi_{(7,10)} + 1\pi_{(9)}}_{\mbox{from}\ \ (L,x)}.
\end{align*}
Apparently, the BS decomposition of $\mfc$ provides an additional pure diagram, $\pi_{(4,10,11)}$, which does not appear in the BS decomposition of $L$.
Nevertheless it still supports the idea of the covering up the middle part of the decomposition of $\be(L)$ by using decompositions of $\be(\mfb)$ and $\be(\mfc)$. 
\end{example}

\begin{example}\label{comp exp}
In the previous example we saw that our approximation for $L$ gives a longer BS chain of the degree sequences than the actual BS chain of the degree sequences via the BS decomposition
of the ideals $\mfa$, $\mfb$, $\mfc$ and $(L,x)$. 

Consider the lex-segment ideal $L=(x^2, xy, xz^2, y^6, y^5z, y^4z^3, y^3z^4, y^2z^5, yz^6, z^9)$ in $R$. Then the colon ideals are $\mfa = L : x = (x,y,z^2)$, $\mfb = L : y = (x, y^5, y^4z, y^3z^3, y^2z^4, yz^5, z^6)$, and $\mfc = L : z = 
(x^2, xy, xz, y^5, y^4z^2, y^3z^3, y^2z^4, yz^5, z^8)$.
The mapping cone for the ideal $\mfc$ requires two cancellations, so we ignore the first two degree sequences. Then, 
\begin{align*}
\be(\mfa) &= \frac{1}{3}\pi_{(1,2,4)} + \frac{1}{3}\pi_{(1,3,4)} + [\mbox{pure diags. of length }<3],\\
\be(\mfb) &= \frac{1}{5}\pi_{(1,6,7)} + \frac{9}{35}\pi_{(1,6,8)} + \frac{2}{7}\mathbf{\pi_{(1,7,8)}} + \frac{1}{2}\mathbf{\pi_{(5,7,8)}} + [\mbox{pure diags. of length }<3],\\
\be(\mfc) &= \cancel{1\pi_{(2,3,4)}} + \cancel{\frac{1}{6}\pi_{(2,3,8)}} + \frac{1}{3}\pi_{(2,6,8)} + \frac{19}{30}\pi_{(2,7,8)} 
+\frac{1}{15}\mathbf{\pi_{(2,7,10)}} + \frac{1}{3}\pi_{(5,7,10)} \\
& \,\,\,\,\,+ [\mbox{pure diags. of length }<3], \mbox{ and }\\ 
\be(L,x) &= [\mbox{pure diags. length}\ 3] + \frac{1}{2}\pi_{(6,8)} + 2\pi_{(7,8)} + 2\pi_{(7,10)} + 1\pi_{(9)}.
\end{align*}

Then, we get the following chain of degree sequences in order to set up the approximate BS decomposition for $L$
\begin{eqnarray*}\be(L) &\approx & \underbrace{(2,3,5) < (2,4,5)}_{\mbox{from} \ \mfa(-1)} < \underbrace{(2,7,8)  < (2,7,9) < {\bf (2,8,9)< (6,8,9)}}_{\mbox{from} \ \mfb(-1)}\\
 &<& \underbrace{(3,7,9)<(3,8,9)<{\bf (3,8,11)}<(6,8,11)}_{\mbox{from} \ \mfc(-1)} < \underbrace{(7,9)< (8,9)<(8,11)< (10)}_{\mbox{from} \ (L,x)}.
\end{eqnarray*}

However, we know that the degree sequences in the decomposition must be a partial ordered chain, so  the ones that violate the partial order are needed to be eliminated. 
From the decomposition of $\be(\mfc)$, we get $(3,7,9)$ as the first degree sequence, 
but $(2,8,9)$ and $(6,8,9)$ cannot be before $(3,7,9)$. So we have to ignore the sequences $(2,8,9)$ and $(6,8,9)$. Then we get an approximate decomposition such as
\begin{align*}
\be(L)&\approx \frac{1}{3}\pi_{(2,3,5)} + \frac{1}{3}\pi_{(2,4,5)} + \al_3\pi_{(2,7,8)} + \al_4\pi_{(2,7,9)}+  
\al_7\pi_{(3,7,9)} \\
&\,\,\,\,\, + \al_8\pi_{(3,8,9)} + \al_9\pi_{(3,8,11)} +\al_{10}\pi_{(6,8,11)} +\frac{1}{2}\pi_{(6,8)} + \frac{1}{2}\pi_{(7,9)} \\
&\,\,\,\,\,+ 2\pi_{(8,9)} + 2\pi_{(8,11)} + 1\pi_{(10)}.
\end{align*}
The BS decomposition of $\be(L)$ is
\begin{align*}
\be(L) &= \frac{1}{3}\pi_{(2,3,5)} + \frac{1}{3}\pi_{(2,4,5)} + \mathbf{\frac{1}{3}\pi_{(2,4,8)}} + \frac{2}{15}\pi_{(2,7,8)} + \frac{1}{10}\pi_{(2,7,9)}\\
&\,\,\,\,\,+  
\frac{1}{2}\pi_{(3,7,9)} + \frac{1}{2}\pi_{(3,8,9)} +\frac{1}{2}\pi_{(6,8,11)} + 
\frac{1}{2}\pi_{(6,8)} + \frac{1}{2}\pi_{(7,9)} \\
&\,\,\,\,\,+ 2\pi_{(8,9)} + 2\pi_{(8,11)} + 1\pi_{(10)}.
\end{align*}
The degree sequence $(3,8,11)$ associated with $(2,7,10)$, which is coming from the decomposition of $\be(\mfc)$, does not show up in the decomposition of $\be(L)$, 
similar to the situation in Example \ref{longer decomp}.
Moreover, for this lex ideal $L$, we realize another different situation. The degree sequence $(2,4,8)$ shows up in the BS chain of degree sequences of $\be(L)$, but $(2-1, 4-1, 8-1) = (1,3,7)$ 
does not appear in any of the decompositions of $\be(\mfa)$, $\be(\mfb)$ and $\be(\mfc)$.

An explanation for that extra degree sequence $(2,4,8)$ might be possible for this example. We see that $(2,4,5)$ is the last degree sequence coming from $\mfa(-1)$ and 
the next degree sequence $(2,7,8)$ is from $\mfb(-1)$. If we assume that there is no other degree sequence between $(2,4,5)$ and $(2,7,8)$, it implies
that simultaneous elimination of the entries in the positions of $\be_{1,4}$ and $\be_{2,5}$ in the Betti diagram of $L$ by the algorithm of BS decomposition.
However, this is not possible because otherwise there would not be a pure diagram of length $2$ in the BS decomposition of $\mfa$.
Hence again by the partial order, it must be $(2,4,5)<{\bf (2,4,8)}<(2,7,8)$.

\end{example}

Examples \ref{longer decomp} and \ref{comp exp} show that the BS decompositions of $\mfa$, $\mfb$, $\mfc$ and $(L,x)$ 
may not be enough to provide the entire chain of degree sequences in the BS decomposition of $L$. Therefore, it is possible that there are some gaps and redundant
degree sequences in the approximation of BS chain of degree sequences of $L$. In view of the explanations, such as the cancellations in mapping cone, the necessity of the order of the chain of the degree sequences, we are able to provide the entire chain of degree sequences in the BS decomposition of $L$. 

\begin{problem} Is it true that the the Boij-S\"oderberg decomposition of a lex ideal $L$ can be described by the decompositions of its colon ideals $\mfa$, $\mfb$, $\mfc$ and $(L,x)$ precisely. That is, in terms of all the pure diagrams and their coefficients?

The relation between BS decompositions of a (Artinian) lex-segment ideal $L$ and the lex ideals $\mfa = L : x$ and $(L,x)$ is pointed out in Theorems \ref{thm1} and \ref{thm2}.
Furthermore, the examples we have observed in this section show that if we know the BS decompositions of the colon
ideals $\mfb = L : y$ and $\mfc=L:z$, then almost the entire BS chain of the degree sequences for the lex-segment ideal $L$ may be revealed.
In other words, we try formalize the full chain of degree sequences of the BS decomposition of the ideal $L$ by using the BS chains of degree sequences of the colon ideals $\mfa$,$\mfb$, $\mfc$ and the lex ideal $(L,x)$. Studying what the observations indicate is a further direction for our research on BS decomposition of lex-segment ideals. 

A natural follow-up work which aims to describe all BS coefficients of a lex ideal $L$ in terms of the coefficients of its colon ideals $\mfa, \mfb, \mfc$
and the larger lex ideal $(L,x)$ may arise at this point, i.e., we narrow our attention on the degree
sequences, that is, pure diagrams. Although the results about $\mfa$ and $(L,x)$ involve the coefficients as well, we do not have a foresight regarding a relation between the coefficients in the decompositions of $\mfb$, $\mfc$ and $L$ based on the observations mentioned in this section.  

\end{problem}

\section*{Acknowledgment}
The author would like to thank to her Ph.D.
adviser Uwe Nagel for proposing the problem about Boij-S\"oderberg decompositions of lex-segment ideals and his contributions and discussions during the course of preparing this manuscript. 
She would also like to thank Daniel Erman for introducing her to the Boij-S\"oderberg Theory during a 2011 Summer Graduate Workshop at MSRI.


\end{document}